\documentclass{amsart}
\usepackage{amssymb}
\usepackage{verbatim}
\usepackage{array}
\usepackage{latexsym}
\usepackage{enumerate}
\usepackage{amsmath}
\usepackage{amsfonts}
\usepackage{amsthm}
\usepackage{color}
\usepackage[english]{babel}

\newtheorem{theorem}{Theorem}[section]
\newtheorem{proposition}[theorem]{Proposition}
\newtheorem{lemma}[theorem]{Lemma}

\newtheorem{example}[theorem]{Example}
\newtheorem{result}[theorem]{Result}
\newtheorem{rem}[theorem]{Remark}

\def\cE{\mathcal E}
\def\cF{\mathcal F}

\def\cQ{\mathcal Q}

\def\cX{\mathcal X}
\def\cY{\mathcal Y}

\def\K{\mathbb{K}}

\def\deg{\mbox{\rm deg}}

\def\div{\mbox{\rm div}}

\def\Div{\mbox{\rm div}}

\def\gg{\mathfrak{g}}

\newcommand{\aut}{\mbox{\rm Aut}}

\newcommand{\ha}{{\textstyle\frac{1}{2}}}
\newcommand{\thi}{\textstyle\frac{1}{3}}

\title{Large odd prime power order automorphism groups of algebraic curves in any characteristic}
\date{}
\author{G\'abor Korchm\'aros and Maria Montanucci}
\begin{document}
\maketitle

\vspace{0.5cm}\noindent {\em Keywords}:
Algebraic curves, algebraic function fields, positive characteristic, automorphism groups.
\vspace{0.2cm}\noindent

\vspace{0.5cm}\noindent {\em Subject classifications}:
\vspace{0.2cm}\noindent  14H37, 14H05.


\begin{abstract} Let $\cX$ be a (projective, geometrically irreducible, nonsingular) algebraic curve of genus $\gg\ge 2$ defined over an algebraically closed field $\K$ of odd characteristic $p\ge 0$, and let $\aut(\cX)$ be the group of all automorphisms of $\cX$ which fix $\K$ element-wise. For any a subgroup $G$ of $\aut(\cX)$ whose order is a power of an odd prime $d$ other than $p$, the bound proven by Zomorrodian  for Riemann surfaces is $|G|\leq 9(\gg-1)$ where the extremal case can only be obtained for $d=3$. We prove Zomorrodian's result for any $\K$. The essential part of our paper is devoted to
extremal $3$-Zomorrodian curves $\cX$. Two cases are distinguished according as the quotient curve $\cX/Z$ for a central subgroup $Z$ of $\aut(\cX)$ of order $3$ is either elliptic, or not.
For elliptic type extremal $3$-Zomorrodian curves $\cX$, we completely determine the two possibilities for the abstract structure of $G$ using deeper results on finite $3$-groups. We also show infinite families of extremal $3$-Zomorrodian curves for both types, of elliptic or non-elliptic. Our paper does not adapt methods from the theory of Riemann surfaces, nevertheless it sheds a new light on the connection between Riemann surfaces and their automorphism groups.
\end{abstract}
    \section{Introduction}
In this paper, $\cX$ stands for a (projective, geometrically irreducible, nonsingular algebraic) curve of genus $\gg\geq 2$ defined over an algebraically closed field $\mathbb{K}$ of any characteristic $p$, and, for an odd prime $d$ different from $p$, $G$ denotes a $d$-subgroup of $\aut(\cX)$, that is, a subgroup  whose order is a power of $d$.

If $\mathbb{K}$ is the complex field, then $\cX$ can be viewed as a Riemann surface, and for this case R. Zomorrodian  proved $|G|\le 9(\gg-1)$; see \cite{zr,zr1}. He also pointed out that the bound is sharp if and only if $\gg-1$ is a power of $d=3$ and $\gg\geq 10$. His approach, inspired by Harvey's work \cite{hw}, was based on the method of Fuchsian groups including Singerman's Theorem, and it was largely used later in the studies of the minimum genus problem and the maximum order problem for other types of automorphism groups; see the survey paper \cite{bcg}.

If $\mathbb{K}$ is any field, especially of positive characteristic, a completely different approach is necessary. However Zomorrodian's bound still holds, see Theorems \ref{proCdic112017} and \ref{lem13jan2018}. The most interesting case arises when the bound is sharp, that is $d=3$, $\gg=3^h+1$, and $|G|=3^{h+2}$ with $h\geq 1$. For a curve $\cX$ for which this occurs, we use the name ``extremal $3$-Zomorrodian curve''. Obviously, $G$ is then  a Sylow $3$-subgroup of $\aut(\cX)$.

In the present paper, extremal $3$-Zomorrodian curves are thoroughly investigated. Our approach relies on Group theory. We look inside the action of $G$ viewed as a permutation group on the points of an extremal $3$-Zomorrodian curve $\cX$, and try to extract as much as possible useful properties regarding both the geometry of $\cX$ and the structure of $G$. Doing so, deeper results on groups whose order is a prime power are helpful and can be used in the following way.

Since the center $Z(G)$ of $G$ is non-trivial,  $Z(G)$ contains some subgroup of order $3$. Such a subgroup $Z$  gives rise to a quotient curve $\bar{\cX}=\cX/Z$ so that the quotient group $\bar{G}=G/Z$ is a subgroup $\aut(\bar{\cX})$. From Proposition \ref{prop29luglio2018}, either $\bar{\cX}$ is elliptic with zero $j$-invariant, or the cover $\cX|\bar{\cX}$ is unramified, and  in the latter case $\bar{\cX}$ is also an extremal $3$-Zomorrodian curve (with Sylow $3$-subgroup $\bar{G})$. Therefore, we are led to work out the former case, that is, to carry out a detailed investigation of  ``elliptic type'' extremal $3$-Zomorrodian curves.

It turns out that the assumption for a group to be a Sylow $3$-subgroup $G$ of an elliptic type extremal $3$-Zomorrodian curve is quite restrictive; see Lemmas \ref{lemC29luglio2018} and \ref{keylemma}. In particular, either $|Z(G)|=3$ or $Z(G)$ is an elementary abelian group of order $9$; see Proposition \ref{prop29luglio2018}. In the former case  $G$ is of maximal nilpotency class, that is $\mathrm{cl}(G)=h+1$ where $|G|=3^{h+2}$; while in the latter case $\mathrm{cl}(G)=h$. Also, $G$ contains a subgroup of index $3$ which is either abelian, or minimal non-abelian; see Lemma \ref{keylemma}. These properties of $G$ together with classical results of Burnside and quite recent results obtained in \cite{Qu} are enough to determine completely the possibilities for $G$ in terms of generators and relations; see Theorems \ref{the25agosto} and \ref{the25agostoA}.

Another important issue is the existence of extremal $3$-Zomorrodian curves, especially of elliptic type. Analogously to what was proven for Riemann surfaces in \cite{zr1}, no extremal $3$-Zomorrodian curve of genus $\gg=4$ exists; see Lemma \ref{nog=4}. The only familiar example of a (non-elliptic type) extremal $3$-Zomorrodian curve is the plane Fermat curve $\cF_9$ of degree $9$ (and genus $28$) which has an automorphism group $G$ of order $243$ isomorphic to $(C_9\times C_9)\rtimes C_3$. Its quotient curve $\bar{\cX}=\cX/Z(G)$ is an elliptic type extremal $3$-Zomorrodian curve of genus $10$; see Example \ref{ex27Aluglio}.

 A complete, positive solution for the existence problem of elliptic type extremal $3$-Zomorrodian curves is given in Section \ref{ss29luglioA}. For every $h\geq 2$, we construct an elliptic type extremal $3$-Zomorrodian curve $\cX$ of genus $\gg(\cX)=3^h+1$ as a degree $3$ Kummer extension of an elliptic curve with vanishing $j$-invariant and prove that a Sylow $3$-subgroup $G$ of $\aut(\cX)$ has center $Z(G)\cong C_3\times C_3$; see Proposition \ref{pro6agosto2018}. Among the four quotient curves arising from the order $3$ subgroups of $Z(G)$, just one is elliptic while the others are (possible isomorphic) elliptic type extremal $3$-Zomorrodian curves of genus $3^{h-1}+1$; see Remark \ref{rem11agosto2018}. Actually, for $h\geq 4$, at least two of the latter three curves are not isomorphic as both possibilities (being of order $3$ or $9$) for the center of a Sylow $3$-subgroup do occur. Using the curve in the order $9$ case, this process repeats as far as $h\geq 5$ and provides at least two non isomorphic elliptic type extremal $3$-Zomorrodian curves of genus $3^{h-1}+1$. We illustrate how to obtain in this way several examples of genus $10, 28$ or $82$; see Examples in Sections \ref{ss29luglioA} and \ref{lowex}.

In Section \ref{ss29luglioB} we exhibit an infinite family of extremal $3$-Zomorrodian curves of non elliptic type.

Our final remark regards the case $d=p$ not considered in the paper. Let $S$ be a $p$-subgroup of $\cX$. If the $p$-rank $\gamma(\cX)$ of $\cX$ is positive then Nakajima's bound yields $|S|\leq 3(\gg(\cX)-1)$, \cite{N} see also \cite[Theorem 11.84]{HKT}, and this bound is attained by an infinite family of curves; see \cite{gkja}.
If $\gamma(\cX)=0$ then $G$ fixes a point of $\cX$, see \cite{gktrans} or \cite[Lemma 11.129]{HKT}, and Stichtenoth's bound gives $|S|\leq 4p/(p-1)^2 \gg$, \cite{stichtenoth1973II}; see also \cite[Theorem 11.78]{HKT}. This case is thoroughly  investigated in \cite{LM,MR}.

\section{Background and Preliminary Results}\label{sec2}
For a finite subgroup $G$ of $\aut(\cX)$, let $\bar \cX$ denote a non-singular model of $\K(\cX)^G$, that is,
a (projective non-singular geometrically irreducible) algebraic curve with function field $\K(\cX)^G$, where $\K(\cX)^G$ consists of all elements of $\K(\cX)$ fixed by every element in $G$. Usually, $\bar \cX$ is called the
quotient curve of $\cX$ by $G$ and denoted by $\cX/G$. The field extension $\K(\cX)|\K(\cX)^G$ is Galois of degree $|G|$.

Since our approach is mostly group theoretical, we prefer to mostly use notation and terminology from Group theory rather than from Function field theory.
\subsection{Background on automorphisms of algebraic curves}
Let $\Phi$ be the cover of $\cX|\bar{\cX}$ where $\bar{\cX}=\cX/G$. A point $P\in\cX$ is a ramification point of $G$ if the stabilizer $G_P$ of $P$ in $G$ is nontrivial; the ramification index $e_P$ is $|G_P|$; a point $\bar{Q}\in\bar{\cX}$ is a branch point of $G$ if there is a ramification point $P\in \cX$ such that $\Phi(P)=\bar{Q}$; the ramification (branch) locus of $G$ is the set of all ramification (branch) points. The $G$-orbit of $P\in \cX$ is the subset
$o=\{R\mid R=g(P),\, g\in G\}$ of $\cX$, and it is {\em long} if $|o|=|G|$, otherwise $o$ is {\em short}. For a point $\bar{Q}$, the $G$-orbit $o$ lying over $\bar{Q}$ consists of all points $P\in\cX$ such that $\Phi(P)=\bar{Q}$. If $P\in o$ then $|o|=|G|/|G_P|$ and hence $\bar{\cQ}$ is a branch point if and only if $o$ is a short $G$-orbit. It may be that $G$ has no short orbits. This is the case if and only if every non-trivial element in $G$ is fixed--point-free on $\cX$, that is, the cover $\Phi$ is unramified. On the other hand, $G$ has a finite number of short orbits.

In this paper we deal with subgroups $G$ of $\aut(\cX)$ whose order is prime to $p$.

Let $\bar{\gg}$ be the genus of the quotient curve $\bar{\cX}=\cX/G$. Since $p\nmid |G|$, the Riemann-Hurwitz
genus formula  is
    \begin{equation}
    \label{eq1}
2\gg-2=|G|(2\bar{\gg}-2)+\sum_{i=1}^s (|G|-\ell_i)
    \end{equation}
    where $\ell_1,\ldots,\ell_s$ denote the size of the short orbits of $G$.
\begin{result}\cite[Theorem 11.56]{HKT}.
\label{res1} If $G$ is abelian then $|G|\leq 4\gg(\cX)+4$.
\end{result}
 Let $\cE$ be an elliptic curve equipped with its group law ``$\bigoplus$''with respect to a point $O\in \cE$. For a point $Q\in \cE$, the translation $\tau_Q$ is the map $P\mapsto Q\bigoplus P$ and the translation group of $\cE$ consists of all translations, and it is a subgroup of $\aut(\cX)$. Since $\bigoplus$ is commutative, any two translations commute. Also, the conjugate of a translation by any automorphism is still a translation. This gives the following well known result.
\begin{result}
\label{res4} The translation group $J(\cE)$ of $\cE$ is a sharply transitive permutation group on $\cE$, and $\aut(\cE)=J(\cE)\rtimes G_P$ for any $P\in\cE$.
\end{result}
Since any genus $2$-curve $\cX$ has an involutory automorphism in the center of $\aut(\cX)$, and the groups of order $9$ are abelian, Result \ref{res1} has the following corollary.
\begin{result}
\label{res15bis} Let $\gg(\cX)=2$. Then $\aut(\cX)$ has no subgroup of order $9$.
\end{result}
\subsection{Background on groups whose order is a power of an odd prime}
 In our proofs we use some basic results on $d$-groups, see \cite{BB1,BB2,BB3,huppertI1967}, together with a corollary to the classification of finite subgroups of the projective linear group $PGL(2,\K)$; see \cite{maddenevalentini1982} or
\cite[Theorem A.8]{HKT}.
\begin{result}
\label{resgroup}\cite[Theorem 2.37 (iii)]{MA}. Let $N$ be a normal subgroup of $G$. Then $G/N$ is abelian if and only if $N$ contains $G'$.
\end{result}
\begin{result} [Burnside basis theorem, {see \cite[Chapter 3, Satz 3.15]{huppertI1967}}]
\label{resburnside} Let $G$ be a $d$-group. Then $G/\Phi(G)$ is an elementary abelian group whose rank is equal to the minimum number of generators of $G$.
\end{result} A non-abelian group is \emph{minimal non-abelian} if each of its proper subgroups is abelian.
\begin{result}
\label{resgroup1}\cite[Lemma 2.2]{Xu2008} Let $G$ be a non-abelian $d$-group. Then $G$ is minimal non-abelian if and only if $G$ can be generated by two elements and $|G'|=d$.
\end{result}
For a group $G$ of order $d^m$ with $d$ prime, let $\{1\}=Z_0< Z_1=Z(G) \le Z_2\le \cdots Z_k \cdots \le Z_n=G$ be the ascending central series of $G$ where $Z_{k+1}/Z_k=Z(G/Z_k)$, for $0\le k \le n$, and $n=\mathrm{cl}(G)$ is the nilpotency class of $G$. Here $\mathrm{cl}(G)\leq m-1$, and if equality holds then $G$ has {\emph{maximal nilpotency class}}.
Let $G$ be of maximal nilpotency class. Then its descending central series $G=K_1(G)\geq K_2(G),\cdots \ge K_k(G) \ge \cdots \ge \{1\}$ with $K_{i+1}(G)=[K_i(G),G]$ has the same size $m-1$, and $Z_i(G)=K_{m-1}(G)$ for $i=0,1,\ldots m-1.$ Furthermore, its characteristic subgroup $$G_1=C_G(K_2(G)/K_4(G))$$
is the \emph{fundamental subgroup} of $G$; see \cite[Definition 14.3]{huppertI1967}.
\begin{result}
\label{resgroup2}\cite[Theorem 2.4]{Xu2008} Let $G$ be a $d$-group of maximal class. Then $[G:G']=d^2$, $G'=\Phi(G)$ and $G$ can be generated by two elements.
\end{result}
\begin{result}
\label{resgroup5}\cite[Theorem 3.5]{Xu2008}. Assume that a non-abelian $d$-group $G$ has an abelian maximal subgroup. Then $G$ is of maximal nilpotency class if and only if either $|Z(G)| = d$ or $[G : G']=d^2$.
\end{result}
\begin{result}[Corollary to Dickson's classification, {see also \cite[Theorem 1]{maddenevalentini1982}}]
\label{res5}
For an odd prime $d$, let $G$ be a $d$-subgroup of $PGL(2,\K)$. If $d\neq p$ then $G$ is cyclic and it has two fixed points in the natural $3$-transitive action of $PGL(2,\K)$ on the projective line over $\K$.
\end{result}
\subsection{Preliminary results on $3$-groups}
Essential ingredients from Group theory in our proofs are a number of deeper results on $3$-groups including classification theorems in terms of generators and relations.
\begin{result}\cite[Corollary 3.5]{Qu}
\label{rescor3.5}Let $G$ be a $3$-group of order at most $3^5.$ If $G$ has a unique minimal non-abelian subgroup then $G$ is isomorphic to one of the following groups.
\begin{itemize}
\item[(i)] Groups of order $3^5$ and of maximal nilpotency class containing an abelian maximal subgroup;
\item[(ii)] $<a,b,c|a^9=b^9=c^3=1,b^{-1}a^{-1}ba=c, c^{-1}a^{-1}ca=a^3,c^{-1}b^{-1}cb=b^{-3}>$.
\end{itemize}
\end{result}
\begin{result}\cite[Lemma 2.12]{Qu}
\label{qu14sep} Let $G$ be a $3$-group of maximal nilpotency class. If its fundamental subgroup $G_1$ is non-abelian then $|G|\geq 3^5$ and $G$ is isomorphic to one of the following non-isomorphic groups.
\begin{itemize}
\item[\rm(i)] $|G|=3^{2e}$ and $G=\langle s_1,s_2,s \mid s_1^{3^e}=s_2^{3^{e-1}}=1, s^3=s_1^{\delta 3^{e-1}}, [s_1,s]=s_2, [s_2,s]=s_2^{-3}s_1^{-3}, [s_2,s_1]=s_1^{3^{e-1}}\rangle$ where $\delta=0,1,2$.
\item[\rm{(ii)}] $|G|=3^{2e+1}$ and $G=\langle s_1,s_2,s \mid s_1^{3^e}=s_2^{3^{e}}=1, s^3=s_2^{\delta 3^{e-1}}, [s_1,s]=s_2, [s_2,s]=s_2^{-3}s_1^{-3}, [s_2,s_1]=s_2^{3^{e-1}} \rangle$ where $\delta=0,1,2$.
\end{itemize}
\end{result}
\begin{result}
\label{resgroup3}\cite[Theorem 3.6]{Qu}. Let $G$ be a $3$-group of order $\ge 3^6$. Then $G$ has a unique minimal non-abelian subgroup whenever $G$ has no abelian maximal subgroup but it has a maximal quotient $\bar{G}$ such that
 \begin{itemize}
 \item[(i)] $\bar{G}$ is of maximal nilpotency class;
 \item[(ii)] $\bar{G}$ has an abelian maximal group.
 \end{itemize}
\end{result}
\begin{result}
\label{nuovo} Let $G$ be a group of order $3^m$ with $m\geq 4$ whose center is an elementary abelian group of order $9$ and contains a subgroup $C$ of order $3$ such that the quotient group $G/C$ is of maximal nilpotency class. Then three of the four quotient groups $G/Z$  with $\{1\}\lneqq Z \lneqq Z(G)$ have center of order $3$, and one has center of order $9$.
\end{result}
\begin{proof}
With the above notation, $Z_2=\{g:gvg^{-1}v^{-1}\in Z_1, \forall v\in G\}.$ Also $n=m-2$ as $G/C$ has order $3^{m-1}$ and  (maximal) nilpotency class $m-2.$ Thus, $|Z_k|=3^{k+1}$ for $1\le k \le m-2$. In particular, $|Z_2|=27$. For an order $3$  subgroup $U$ of $Z_1$, let $\hat{U}=\{u: uvu^{-1}v^{-1}\in U,\forall v\in G\}.$ Obviously, $Z_1\le \hat{U}\le Z_2$. Moreover, $g\in \hat{U}$ if and only if the coset $gU$ is in $Z(G/U)$. Take  $g\in Z_2\setminus Z_1$ together with $v\in G$ such that $gv\neq vg$, and set $t=gvg^{-1}v^{-1}$. Obviously, $t\in Z_1$ is a nontrivial element. Take for $U$ the subgroup generated by $t$. Then $\hat{U}\gneqq Z_1$ whence $|\hat{U}|\geq 27$. Since $|Z_2|=27$ and $\hat{U}\le Z_2$, this yields $\hat{U}=Z_2$. Therefore, $|Z(G/U)|=9$, and  $Z_2=\{g:gvg^{-1}v^{-1}\in U, \forall v\in G\}$. Now, choose an order $3$ subgroup $D$ of $Z_1$ other than $U$. Obviously, $U\cap D=\{1\}$. Since  $\hat{D}=\{u: uvu^{-1}v^{-1}\in D,\forall v\in G\}$ and $\hat{D}\leq Z_2$, it turns out that
$\hat{D}=\{u: uvu^{-1}v^{-1}=1,\forall v\in G\}$, that is, $\hat{D}\le Z_1$. Thus $Z(G/D)=Z_1/D$ whence $|Z(G/D)|=3$.
\end{proof}
The following classical result is due to Blackburn \cite{Bl1}, see also \cite[Satz 14.17, 14.22]{huppertI1967}.
\begin{result}
\label{res14.17}
Let $G$ be a $3$-group of maximal nilpotency class. Then its fundamental group $G_1$ is metacyclic with $\mathrm{cl}(G_1)\le 2$, and $(G')'=\{1\}$. Furthermore, each of the maximal subgroups of $G$ other than $G_1$ is also of maximal nilpotency class.
\end{result}
A refinement of Blackburn's classification \cite{Bl1}, see also \cite[Lemma 2.11]{Qu} is given in the following result.
\begin{result}
\label{resbl1} Let $G$ be a $3$-group of maximal nilpotency class and of order $\ge 3^5$. If
\begin{itemize}
\item[(i)]
$G$ has an abelian subgroup $H$ of index $3$,
\item[(ii)] $G$ can be generated by two elements,
\item[(iii)] every element in $G\setminus H$ has order three,
\end{itemize}
then
\begin{equation}
\label{eq11sep}
G=
\begin{cases}
\langle s_1,s_2,\beta|s_1^{3^e}=s_2^{3^{e}}=\beta^3=1,[s_1,\beta]=s_2,[s_2,\beta]=s_2^{-3}s_1^{-3},[s_1,s_2]=1\rangle;{\mbox{ for $|G|=3^{2e+1}$}};\\
\langle s_1,s_2,\beta|s_1^{3^e}=s_2^{3^{e-1}}=\beta^3=1,[s_1,\beta]=s_2,[s_2,\beta]=s_2^{-3}s_1^{-3},[s_1,s_2]=1\rangle;{\mbox{ for $|G|=3^{2e}$}};\\
\end{cases}
\end{equation} where $H=\langle s_1,s_2 \rangle.$
\end{result}
\begin{proof}
From (ii), we have $[G:\Phi(G)]=9$, and hence the maximal subgroups of $G$ are exactly its subgroups of index $3$. Since $G$ is of maximal nilpotency class of order $\ge 243$, its maximal abelian subgroups are not of maximal nilpotency class. Therefore, (i) together with Result \ref{res14.17} yield that $H$ coincides with the fundamental subgroup $G_1$ of $G$. Now,  from \cite[Lemma 2.11]{Qu} one of the following cases occurs:
\vspace{0.2cm}
\begin{itemize}
\item[(A)] $|G|=3^{2e+1}$ with $e \geq 2$,
\begin{itemize}
\item[(1)] $ G=\langle s_1,s_2,\beta \mid s_1^{3^e}=s_2^{3^e}=\beta^3=1, [s_1,\beta]=s_2, [s_2,\beta]=s_2^{-3} s_1^{-3}, [s_1,s_2]=1 \rangle,$
 where $G_1=\langle s_1,s_2 \rangle$ and $g^3=1$ for all $g \in G \setminus G_1$;
\item[(2)] $G=\langle s_1,s_2,\beta \mid s_1^{3^e}=s_2^{3^e}=1, \beta^3=s_2^{3^{e-1}}, [s_1,\beta]=s_2, [s_2,\beta]=s_2^{-3} s_1^{-3}, [s_1,s_2]=1 \rangle, $ where $G_1=\langle s_1,s_2 \rangle$ and $\langle g^3 \rangle = \langle s_2^{3^{e-1}}\rangle$ for every $g \in G \setminus G_1$;
\item[(3)] $G=
 \langle s_1,s_2,\alpha,\beta \mid s_1^{3^e}=s_2^{3^{e-1}}=\beta^3=1,\alpha^3=s_1^{-3} s_2^{-1} s_1^{3^{e-1}},
 [\alpha,\beta]=s_1, [s_1,\beta]=s_2,[s_2,\beta]=s_2^{-3}s_1^{-3}, [s_1,\alpha]=[s_1,s_2]=1 \rangle, $
     with $G_1= \langle \alpha, s_1 \rangle= \langle \alpha,s_1,s_2 \rangle$.
\end{itemize}
\end{itemize}
\vspace{0.2cm}
\begin{itemize}
\item[(B)] $|G|=3^{2e}$ with $e \geq 2$
\begin{itemize}
 \item[(1)] $G=\langle s_1,s_2,\beta \mid s_1^{3^e}=s_2^{3^{e-1}}=\beta^3=1, [s_1,\beta]=s_2, [s_2,\beta]=s_2^{-3} s_1^{-3}, [s_1,s_2]=1 \rangle,$ where $G_1=\langle s_1,s_2 \rangle$ and $g^3=1$ for all $g \in G \setminus G_1$, or
 \item[(2)] $G= \langle s_1,s_2,\beta \mid s_1^{3^e}=s_2^{3^{e-1}}=1, \beta^3=s_2^{3^{e-1}}, [s_1,\beta]=s_2, [s_2,\beta]=s_2^{-3} s_1^{-3}, [s_1,s_2]=1 \rangle,$ with $G_1=\langle s_1,s_2 \rangle$ and $\langle g^3 \rangle = \langle s_1^{3^{e-1}}\rangle$ for every $g \in G \setminus G_1$
 \item[(3)] $G=\langle s_1,s_2,\alpha,\beta \mid s_1^{3^{e-1}}=s_2^{3^{e-1}}=\beta^3=1,\alpha^3=s_1^{-3} s_2^{-1} s_2^{\nu 3^{e-1}}, [\alpha,\beta]=s_1, [s_1,\beta]=s_2,[s_2,\beta]=s_2^{-3}s_1^{-3},[s_1,\alpha]=[s_1,s_2]=1 \rangle,$ where $\nu=1,2$, $G_1= \langle \alpha,s_1,s_2 \rangle$ and if $e>2$ or $\nu=2$ then $G_1=\langle \alpha,s_2 \rangle$.
 \end{itemize}
 \end{itemize}
 Clearly, Cases A(2) and B(2) are not possible by (iii) since no element $g \in G \setminus G_1$  may have order $3$ by $\langle g^3 \rangle = \langle s_1^{3^{e-1}}\rangle$ and $\langle g^3 \rangle = \langle s_2^{3^{e-1}}\rangle$, respectively. For Cases A(3) and B(3), we exhibit an element $g \in G \setminus G_1$ such that $g$ has order $9$. This will imply that Cases A(1) and B(1) cannot occur as well, so that Result \ref{resbl1} is true. Let $g=\alpha\beta$. Then $g$ has order $9$ whenever  $(\alpha\beta)^3=s_1^{3^{e-1}}$ for Case A(3) and $(\alpha\beta)^3=s_2^{\nu 3^{e-2}}$ for Case B(3). In both cases, the condition is equivalent to
 \begin{equation}
 \label{eqsep14}
 (\alpha\beta)^3=s_2 s_1^3 \alpha^3.
 \end{equation}
 The proof of (\ref{eqsep14}) is carried out in several steps.
 \subsubsection{Step 1}
  \begin{equation}
  \label{primaeq}
  \beta^{-1}s_2^{-1}\beta s_2^{-1}=\beta s_2 \beta^2 .
  \end{equation} From $[s_1,\beta]=s_2$ and $\beta^3=1$,
    $$\beta^{-1}s_2^{-1}\beta s_2^{-1} =\beta^{-1}(s_1^{-1}\beta^{-1}s_1\beta)^{-1}\beta(s_1^{-1}\beta^{-1}s_1\beta)^{-1}=\beta^{-2}s_1^{-1}\beta^2s_1=\beta s_1^{-1}\beta^{-1}s_1$$
and $$\beta s_2 \beta^2 =\beta(s_1^{-1}\beta^{-1}s_1\beta)\beta^2=\beta s_1^{-1}\beta^{-1}s_1$$
whence (\ref{primaeq}) follows.
\subsubsection{Step 2}
\begin{equation}
\label{secondaeq} \beta s_2 \beta^2=s_1^3 s_2.
\end{equation}
From $[s_2,\beta]=s_2^{-3} s_1^{-3}$, we have $s_1^3=\beta^{-1} s_2^{-1} \beta s_2^{-2}.$ Combining this with Equation \eqref{primaeq} gives $$s_1^3 s_2=(\beta^{-1} s_2^{-1} \beta s_2^{-2})s_2=\beta^{-1}s_2^{-1}\beta s_2^{-1}=\beta s_2 \beta^2,$$
which shows (\ref{secondaeq}). Since $\alpha, s_1$ and $s_2$ commute, (\ref{secondaeq}) yields
\begin{equation}
\label{terzaeq} s_1^3 s_2 \alpha^2=\alpha^2 s_1^3 s_2 = \alpha^2  \beta s_2 \beta^2.
\end{equation}
\subsubsection{Step 3}
From $[\alpha,\beta]=s_1$ and $[s_1,\beta]=s_2$, we infer
$\alpha\beta s_1^{-1}=\beta \alpha$ and $s_1s_2 \beta^{-1}=\beta^{-1} s_1=\beta^2 s_1.$ Hence, $$\alpha^2 \beta s_2 \beta^2=\alpha(\alpha \beta s_1^{-1})s_1 s_2 \beta^{-1}=\alpha(\beta \alpha)s_1 s_2 \beta^{-1}=\alpha\beta\alpha (s_1 s_2 \beta^{-1})
=\alpha\beta\alpha(\beta^2 s_1)=(\alpha\beta)^2 \beta s_1.$$ Multiply both sides by $\alpha$. Since $\alpha,s_1$ and $s_2$ commute pairwise, Equation \eqref{terzaeq} gives
\begin{equation}
\label{quartaeq} (\alpha \beta)^2 \beta \alpha s_1=(\alpha \beta)^2 \beta s_1 \alpha =s_1^3 s_2 \alpha^3. \end{equation}
Furthermore, $\alpha \beta s_1^{-1}=\beta \alpha$ yields $(\alpha \beta)^2 \beta \alpha s_1=(\alpha \beta)^3$.
\subsubsection{Step 4}
If Case A(3) occurs then $\alpha^3=s_1^{-3} s_2^{-1} s_1^{3^{e-1}}$ whence
$s_1^3 s_2 \alpha^3=s_1^3 s_2 (s_1^{-3} s_2^{-1} s_1^{3^{e-1}})=s_1^{3^{e-1}}.$
 If Case B(3) occurs then $\alpha^3=s_1^{-3} s_2^{-1} s_2^{\nu 3^{e-1}}$ whence
 $s_1^3 s_2 \alpha^3=s_1^3 s_2(s_1^{-3} s_2^{-1} s_2^{\nu 3^{e-1}})=s_2^{\nu 3^{e-1}}$.
\subsubsection{Step 5}
Finally, Claim (\ref{eqsep14}) follows from Step 4 and Equation \eqref{quartaeq}.
\end{proof}
In \cite{Qu}, $3$-groups with a unique minimal non-abelian subgroup of index $3$ are classified. Let $G$ be such a group with its minimal non-abelian subgroup $H$ of index $3$. By Result \ref{resgroup1}, $|H'|=3$.
With this notation, \cite[Lemma 2.17, Theorems 3.8, 3.9]{Qu} have the following corollary.
\begin{result}
\label{qu24agosto} Let $G$ be a $3$-group of order $\ge 3^5$ which has a unique minimal non-abelian subgroup $H$ of index $3$.
\begin{itemize}
\item[(i)]
If $H$ is metacyclic then $G$ is of maximal nilpotency class, and the converse also holds.
\item[(ii)]
If $H$ is not metacyclic, and $G\setminus H$ contains at least $\frac{4}{9}|G|$ elements of order $3$, and $G/H'$ is isomorphic to the group in Result \ref{resbl1}, then
$$G=\begin{cases}
\langle s_1,s_2,\beta,x|s_1^{3^n}=s_2^{3^{n-1}}=x^3=1,\beta^3=x^2, [s_1,\beta]=s_2,[s_2,\beta]=s_2^{-3}s_1^{-3},[s_1,s_2]=x, [x,s_1]=[x,s_2]=1\rangle;\\
{\mbox{ for $|G|=3^{2n+1},\,e \geq 3$}};\\
\langle s_1,s_2,\beta,x|s_1^{3^n}=s_2^{3^n}=x^3=1,\beta^3=x^2, [s_1,\beta]=s_2,[s_2,\beta]=s_2^{-3}s_1^{-3},[s_1,s_2]=x, [x,s_1]=[x,s_2]=1\rangle;\\
{\mbox{ for $|G|=3^{2n+2},\,n\geq 2$}};
\end{cases}
$$
where $H'=\langle x \rangle$.
\end{itemize}
\end{result}
\begin{proof}
Since $G$ is assumed to have a unique minimal non-abelian subgroup $H$, Claim (i) follows from \cite[Theorem 3.7]{Qu}. Therefore, $H$ is assumed to be non-metacyclic. From \cite[Theorems 3.8, 3.9]{Qu}, $G$ is one of the non-isomorphic groups below where $k=0,1,2$ and $\nu=1,2$.
\begin{itemize}
\item[\rm{(I)}] $|G|=3^{2e+2}$,
\begin{itemize}
\item[\rm{(1)}] $G=\langle s_1,s_2,\beta,x \mid s_1^{3^e}=s_2^{3^e}=x^3=1,\beta^3=x^k,[s_1,\beta]=s_2, [s_2,\beta]=s_2^{-3} s_1^{-3},[s_1,s_2]=x,$ $ [x,s_1]=[x,\beta]=1\rangle$,
\item[\rm{(2)}] $G=\langle s_1,s_2,\beta,x \mid s_1^{3^e}=s_2^{3^e}=x^3=1, \beta^3=s_2^{3^{e-1}}x^k, [s_1,\beta]=s_2, [s_2,\beta]=s_2^{-3} s_1^{-3}, [s_1,s_2]=x, $ $[x,s_1]=[x,\beta]=1\rangle$,
\item[\rm{(3)}] $G=\langle s_1,s_2,\alpha,\beta,x \mid s_1^{3^e}=s_2^{3^{e-1}}=x^3=1, \beta^3=x^k, \alpha^3=s_1^{-3}s_2^{-1}s_1^{3^{e-1}},[\alpha,\beta]=s_1, [s_1,\alpha]=x,  $ $ [s_1,\beta]=s_2, [s_2,\beta]=s_2^{-3}s_1^{-3}, [s_1,s_2]=[x,\alpha]=[x,\beta]=1\rangle$,
\item[\rm{(4)}] $G=\langle s_1,s_2,\alpha,\beta,x \mid s_1^{3^e}=s_2^{3^{e-1}}=x^3=1, \beta^3=x^k, \alpha^3=s_1^{-3}s_2^{-1}s_1^{3^{e-1}}x,[\alpha,\beta]=s_1, [s_1,\alpha]=x, $ $[s_1,\beta]=s_2, [s_2,\beta]=s_2^{-3}s_1^{-3}, [s_1,s_2]=[x,\alpha]=[x,\beta]=1\rangle$;
\end{itemize}
\item[\rm{(II)}] $|G|=3^{2e+1}$,
\begin{itemize}
\item[\rm{(1)}] $G=\langle s_1,s_2,\beta,x \mid s_1^{3^e}=s_2^{3^{e-1}}=x^3=1, \beta^3=x^k, [s_1,\beta]=s_2, [s_2,\beta]=s_2^{-3}s_1^{-3}, [s_1,s_2]=x, $ $[x,s_1]=[x,\beta]=1 \rangle$,
\item[\rm{(2)}] $G=\langle s_1,s_2,\beta,x \mid s_1^{3^e}=s_2^{3^{e-1}}=x^3=1, \beta^3=s_1^{3^{e-1}}x^k, [s_1,\beta]=s_2, [s_2,\beta]=s_2^{-3}s_1^{-3}, [s_1,s_2]=x, $ $[x,s_1]=[x,\beta]=1 \rangle$,
\item[\rm{(3)}] $G=\langle s_1,s_2,\alpha,\beta,x \mid s_1^{3^{e-1}}=s_2^{3^{e-1}}=x^3=1, \beta^3=x^k, \alpha^3=s_1^{-3}s_2^{-1}s_2^{\nu 3^{e-2}},[\alpha,\beta]=s_1, [s_1,\alpha]=x,$ $ [s_1,\beta]=s_2, [s_2,\beta]=s_2^{-3}s_1^{-3}, [s_1,s_2]=[x,\alpha]=[x,\beta]=1\rangle$,
\item[\rm{(4)}] $G=\langle s_1,s_2,\alpha,\beta,x \mid s_1^{3^{e-1}}=s_2^{3^{e-1}}=x^3=1, \beta^3=x^k, \alpha^3=s_1^{-3}s_2^{-1}s_2^{\nu 3^{e-2}}x, [\alpha,\beta]=s_1, [s_1,\alpha]=x, [s_1,\beta]=s_2, [s_2,\beta]=s_2^{-3}s_1^{-3}, [s_1,s_2]=[x,\alpha]=[x,\beta]=1\rangle$.
\end{itemize}
\end{itemize}
Now, the proof is performed in several steps.
\subsubsection{Step 1} From \cite[Theorem 3.6]{Qu} $\bar G=G/H^\prime$ is a $3$-group of maximal class with an abelian maximal subgroup $\bar H=H/H^\prime$. Also, $\bar G= G/ {H^\prime}$ is isomorphic to the group in Result \ref{resbl1}. The proof of \cite[Theorems 3.8, 3.9]{Qu} shows that either Case (I)(1) or Case (II)(1) occurs for $G$ as they are the unique possibilities for $G$ such that $\bar G$ is one of the groups in Result \ref{resbl1}. In both cases, $\langle s_1,s_2\rangle$ is a minimal non-abelian subgroup, and hence $H=\langle s_1,s_2\rangle$.
It remains to prove that
Cases $k=0,1$ cannot actually occur. For this purpose, an inductive argument on $|G|$ is used.

We begin with Case (I)(1). First, let $|G|=3^{2e+2}$ with $e \geq 2$. Observe that $x \in Z(G)$ as $x$ commutes with both $s_1$ and $\beta$ by $[x,\beta]=[x,s_1]=1$.
\subsubsection{Step 2} We prove $[x,s_2]=1$. Since
$$[x,s_2]=[x,[s_1,\beta]]=x^{-1}(s_1^{-1}\beta^(-1)s_1\beta)^{-1}x(s_1^{-1}\beta^{-1}s_1\beta)= x^{-1}\beta^{-1}s_1^{-1}\beta(s_1 x s_1^{-1})\beta^{-1} s_1 \beta,$$
$[x,s_1]=1$ yields
$[x,s_2]=x^{-1}\beta^{-1} s_1^{-1} (\beta x \beta^{-1}) s_1\beta.$
From this and $[x,\beta]=1$, $[x,s_2]=x^{-1} \beta^{-1} (s_1^{-1} x s_1)\beta=1.$
\subsubsection{Step 3} We prove that $s_2^{3^{e-1}}\in Z(G)$.
From  $[s_1,s_2]=[s_2,s_1]^{-1}=x^{-1}$ and $x \in Z(G)$, $s_1$ commutes with $[s_1,s_2]$. Hence from \cite[Hilfssatz 1.3 (a)]{huppertI1967},
\begin{equation}
\label{eq26sep}
[s_2^{3^{e-1}},s_1]=[s_2,s_1]^{3^{e-1}}=(x^{-1})^{3^{e-1}}=1.
\end{equation}
Since both subgroups $\langle s_1^3,s_2 \rangle$ and $\langle s_2^3,s_1 \rangle$ are proper subgroups of $H$,
the assumption on $H=\langle s_1,s_2 \rangle$ to be minimal non-abelian yields $[s_1,s_2^3]=1$ and $[s_2,s_1^3]=1$. From $\beta s_2 \beta=s_2^{-2} s_1^{-3}=s_1^{-3} s_2^{-2}$,  $$\beta^{-1} s_2^{3^{e-1}} \beta=(s_2^{-2} s_1^{-3})^{3^{e-1}}=(s_2^{-2})^{3^{e-1}} (s_1^{-3})^{3^{e-1}}=(s_2^{3^{e-1}})^{-2},$$ which is equal to $s_2^{3^{e-1}}$ by $(s_2^{3^{e-1}})^3=1$. This together with (\ref{eq26sep}) yield $s_2^{3^{e-1}} \in Z(G)$. Since $x\neq s_2^{3^{e-1}}$, we also have that $s_2^{3^{e-1}}$ is another generator of $Z(G)$.
\subsubsection{Step 4}
The quotient group $\tilde G= G/\langle s_2^{3^{e-1}} \rangle$ is given by
$$\tilde G=\langle \tilde{s}_1,\tilde{s}_2,\tilde{\beta},\tilde{x} \mid \tilde{s}_1^{3^e}=\tilde{s}_2^{3^{e-1}}=\tilde{x}^3=1, \tilde{\beta}^3=\tilde{x}^k, [\tilde{s}_1,\tilde{\beta}]=\tilde{s}_2, [\tilde{s}_2,\tilde{\beta}]=\tilde{s}_2^{-3}\tilde{s}_1^{-3}, [\tilde{s}_1,\tilde{s}_2]=\tilde{x}, [\tilde{x},s_1]=[\tilde{x},\beta]=1 \rangle,$$
and hence $\tilde{G}$ satisfies Case (II)(1) with the same value of $k$ as in $G$. From $ s_2^{3^{e-1}}\in H$, in the natural homomorphism $G\mapsto \bar{G}$, every element of order $3$ in $G\setminus H$ produces an element in $\tilde{G}$ of the same order $3$. Therefore, $\tilde{G} \setminus \tilde{H}$ contains at least as many as $4|\tilde{G}|/9$ elements of order $3$.
\subsubsection{Step 5}
With some formal changes, the above arguments apply to Case (II)(1). Therefore, if $G$ satisfies Case (II)(1) with a given value of $k$ then $\tilde{G}=G/\langle s_1^{3^{e-1}} \rangle$ satisfies Case (I)(1), and it has at least as many as $4|\tilde{G}|/9$ elements in $\tilde{G} \setminus \tilde{H}$.
\subsubsection{Step 6}
Therefore, with an inductive argument on $e$, it suffices to rule out the cases $k=0,1$ when $|G|$ is assumed to be as small as possible, namely $|G|=3^6$, that is, $e=2$. A MAGMA aided computation shows that if $|G|=3^6$ and $k=0$ then $G \cong SmallGroup(729,49)$ while if $k=1$ then $G \cong SmallGroup(729,54)$. In the former case $G \setminus H$ contains $162<324=4|G|/9$ elements of order $3$, a contradiction. In the latter case $G \setminus H$ contains no elements of order $3$, again a contradiction.
\end{proof}

\subsection{Some preliminary results on $3$-groups of automorphisms of elliptic curves}
\label{pre}
Assume that $p\neq 3$.
Let $\bar{\cE}$ be the elliptic curve of homogenous equation $X^3+Y^3+Z^3=0$. Then $\bar{\cE}$ has zero $j$-invariant, $P=(-1,0,1)$ is an inflection point of $\bar{\cE}$, and the map $(X,Y,Z)\to (X,\varepsilon Y,Z)$ with a primitive cubic root of unity $\varepsilon$ is an order $3$ automorphism $\bar{\alpha}$ of $\bar{\cE}$ fixing $P$. Actually, $\bar{\alpha}$ has two more fixed points on $\bar{\cE}$, namely
$P_1=(-\varepsilon, 0,1)$ and $P_2=(-\varepsilon^2,0,1)$. The map $(X,Y,Z)\to (\varepsilon X, \varepsilon^2 Y,Z)$ is another order $3$ automorphism $\beta$ of $\bar{\cE}$. Since $\bar{\alpha}$ and $\bar{\beta}$ commute, $\bar{\beta}$ preserves the pointset $\pi=\{P,P_1,P_2\}$. Let $\bar{J}(\bar{\cE})$ denote the translation group of $\bar{\cE}$. Then $\bar{\beta}$ belongs to $\bar{J}(\bar{\cE})$ but $\bar{\alpha}$ does not.
Choose a $3$-subgroup $\bar{G}$ of $\aut(\bar{\cE})$ such that some element of order $3$ of $\bar{G}$ has a fixed point on $\bar{\cE}$. Since $\aut(\bar{\cX})$ acts transitively on the points of $\bar{\cE}$, such a fixed point may be assumed to be $P$. Then
\begin{itemize}
\item[(i)] $|\bar{G}|=3^{h+1}$ with $h\ge 1$;
\item[(ii)] $|\bar{G}_P|=3$.
\end{itemize}
Furthermore, $\bar{G}_P=\langle \bar{\alpha}\rangle$ and $\bar{\beta}\in Z(\bar{G)}$. More precisely, $Z(\bar{G})=\langle \bar{\beta} \rangle$ as $Z(\bar{G})$ preserves $\pi$ and (ii) holds.
\begin{lemma}
\label{lem28jun} $\bar{G}=\bar{H}\rtimes \bar{G}_P$ where $\bar{H}=\bar{G}\cap J(\bar{\cE})$.
 \end{lemma}
 \begin{proof} If $|\bar{G}|=9$ then $\bar{G}=\langle \bar{\beta} \rangle \times \bar{G}_P $ an the claim follows for $h=1$ by $\bar{\beta}\in J(\bar{\cE}).$
 Let $\bar{W}=\bar{G}\cap J(\bar{\cE})$. Then $\bar{W}$ is a nontrivial normal subgroup and $\hat{G}=\bar{G}/\bar{W}$ is a subgroup of $\aut(\hat{\cE})$ where the quotient curve
 $\hat{\cE}=\bar{\cE}/\bar{W}$ is also elliptic as no nontrivial element in $J(\bar{\cE})$ fixes a point in $\bar{\cE}$. By induction on $h$, assume that $\hat{G}=\hat{T}\rtimes \hat{G}_{\hat{P}}$ where $\hat{T}<J(\hat{\cE})$ and $\hat{P}$ is the (unique) point of $\hat{\cE}$ lying under $P$ in the cover $\bar{\cE}|\hat{\cE}$. Let $\bar{T}$ be the subgroup of $\bar{G}$ containing $\bar{W}$ such that $\hat{T}=\bar{T}/\bar{W}$. Then no nontrivial element in $\bar{T}$ fixes a point of $\bar{\cE}$. Furthermore,
 $|\bar{T}|=|\bar{H}|$, and $\bar{T}$ is a normal subgroup of $\bar{G}$. Therefore, $\bar{G}=\bar{T}\rtimes \bar{G}_P$. Now, let $o$ be the $\bar{T}$-orbit of $P$, and take any point $\bar{P}'$ from $o$. Then $\bar{P}'=\bar{t}({P})$ for some $\bar{t}\in \bar{T}$. On the other hand, $J(\bar{\cE})$ contains an element $\bar{j}$ such that $\bar{j}(\bar{P}')={P}$. Then $\bar{j}\bar{t}({P})={P}$ whence $\bar{j}\in \bar{G}$. Since $|o|=|\bar{T}|$ this yields that $\bar{G}\cap J(\bar{\cE})$ contains at least as many as $|\bar{T}|=|\bar{H}|$ elements. As $\bar{G}_P\cap J(\bar{\cE})$ is trivial, this shows that $\bar{G}\cap J(\bar{\cE})=\bar{T}$ whence the claim follows.
\end{proof}
\begin{lemma}
\label{lem30luglio2018} The centralizer $C_{\bar{G}_P}$ of $\bar{G}_P$ in $\bar{G}$ is an elementary abelian group of order $9$.
\end{lemma}
\begin{proof} Since $\bar{G}_{{P}}=\langle \bar{\alpha}\rangle$ and $\bar{\alpha}$ fixes exactly three points, namely those in $\pi$,  $C_{\bar{G}}(\bar{G}_{{P}})$ preserves $\pi$. This together with $\bar{G}_{{P}}$ yield $|C_{\bar{G}}(\bar{G}_{{P}})|=3$. Therefore $C_{\bar{G}}(\bar{G}_{{P}})$ is the direct product of $\bar{G}_{{P}}$ and $Z(\bar{G})=\langle \bar{\beta} \rangle$.
\end{proof}
The possibilities for the structure of $\bar{H}=\bar{G}\cap J(\bar{\cE})$ in Lemma \ref{lem28jun} are rather restricted.
\begin{proposition}
\label{prop18jul2018PG} If $h$ is even, say $2n$, then $\bar{H}\cong C_{3^n}\times C_{3^n}$. If $h$ is odd, say $2n-1$, then  $\bar{H}\cong C_{3^n}\times C_{3^{n-1}}$.
\end{proposition}
\begin{proof} Since $\bar{H}$ is a subgroup of $J(\bar{\cE})$, we have  $\bar{H}=\bar{U}\times \bar{V}$ where $\bar{U}\cong C_{3^i}\times C_{3^j}$ with $i\ge j$ and $i+j=h$. Let $\bar{T}=\bar{\alpha} \bar{U} \bar{\alpha}^{-1}$.
Consider the set of all elements $\bar{u}\in U$ such that $\bar{\alpha} \bar{u} \bar{\alpha}^{-1}\in \bar{U}$. Actually, this set $\bar{W}$ is a subgroup of $\bar{U}$. Since $\bar{U}$ is cyclic, $\bar{W}$ is also cyclic and it consists of all elements $\bar{u}\in \bar{U}$ whose order divides $|\bar{W}|$. In particular, since $\bar{u}$ and $\bar{\alpha} \bar{u} \bar{\alpha}^{-1}$ have the same order, $\bar{u}\in \bar{W}$ implies $\bar{\alpha} \bar{u} \bar{\alpha}^{-1}\in \bar{W}$, that is, the subgroup of $\bar{G}$ generated by $\bar{W}$ together with $\bar{\alpha}$ is the semidirect product $\bar{W}\rtimes \langle \bar{\alpha} \rangle$.

If $\bar{W}$ is trivial then $\bar{U}\cap \bar{T}=\{\bar{1}\}$ and hence
$\bar{U}\times \bar{T}$ is a subgroup of $\bar{H}$, whence $\bar{H}=\bar{U}\times \bar{T}$ and the first claim follows. Otherwise, $\bar{W}\rtimes \bar{G}_{{P}}$ is a proper subgroup of $\bar{G}$. Since $\bar{W}$ is cyclic, we have
$|\bar{W}|=3$. Therefore $\bar{W}\rtimes \bar{G}_P$ has order $9$ and is abelian. Thus $\bar{W}\rtimes \bar{G}_P=\bar{W}\times \bar{G}_P$. Therefore a generator of $\bar{W}$ commutes with $\bar{\alpha}$, and hence $\bar{W}=\langle \bar{\beta} \rangle$ showing that $\bar{W}=Z(\bar{G})$. Look at the quotient group $\hat{H}=\bar{H}/Z(\bar{G})$. Since $Z(\bar{G})=\bar{W}\le \bar{U}$, we have that $\hat{U}=\bar{U}/Z(\bar{G})$ is a subgroup of $\hat{H}$ of order
$3^{i-1}$. The same holds for the subgroup $\hat{T}=\bar{T}/Z(\bar{G})$ of order $3^{i-1}$. Let $\hat{s}\in \hat{U}\cap \hat{T}$. If $\hat{s}$ is nontrivial then there exist $\bar{u},\bar{t}\in \bar{U}\setminus \bar{W}$ such that
$\bar{u}\bar{t}^{-1}\in \bar{W}$. If $\bar{u}=\bar{t}$ then $\bar{u}=\bar{\alpha} \bar{u}' \bar{\alpha} ^{-1}$ for some $\bar{u}'\in \bar{U}$. This implies  $\bar{u}'\in \bar{W}$ whence $\bar{u}'\in Z(\bar{G})$ and $\bar{u}=\bar{u}'\in Z(\bar{G})$, a contradiction. Similarly, if $\bar{u}=\bar{t}\bar{\beta}$ (or $\bar{u}=\bar{t}\bar{\beta}^2$) then
$\bar{u}=\bar{\alpha} \bar{u}' \bar{\alpha}^{-1}\bar{\beta}$ for some $\bar{u}'\in \bar{W}$. Since $\bar{\beta}\in \bar{W}$, this implies $\bar{u}=\bar{u}'\bar{\beta}$ whence $\bar{u}\in \bar{W}=Z(\bar{G})$, a contradiction.
Therefore,  $\hat{U}\cap \hat{T}$ is trivial, and $|\hat{H}|=|\hat{U}||\hat{T}|$ whence $|\hat{H}|=3^{2(i-1)}$. On the other hand, $|\hat{H}|=\frac{1}{3}|\bar{H}|=3^{i+j-1}$. Thus
$2(i-1)\leq i+j-1$ whence $i-1\le j$. Since $i\geq j$, this yields either $i=j$ or $j=i-1$.
\end{proof}
\begin{lemma}
\label{lemB29luglio2018} If $h\geq 2$ then the following hold.
\begin{itemize}
\item[(i)] $\bar{G}$ is of maximal nilpotency class.
\item[(ii)] $\bar{G}$ can be generated by two elements.
\item[(iii)] Every element in $\bar{G}\setminus \bar{H}$ has order $3$.
\item[(iv)] In terms of generators and relations, $\bar{G}$ is given in Result \ref{resbl1}.
\end{itemize}
\end{lemma}
\begin{proof} As $\bar{G}$ is non-abelian and  $\bar{H}$ is an abelian maximal subgroup of $\bar{G}$, Result \ref{resgroup5} applies. Since  $|Z(\bar{G})|=3$, this gives the first claim whence the second claim follows by Result \ref{resgroup2}. To prove the third claim, apply the Hurwitz genus formula for $\bar{G}$. From $\bar{\alpha}\in  \bar{G}$, $\bar{G}$  has $k\geq 1$ short orbits and Hurwitz genus formula gives
$$0=2\gg(\bar{\cE})-2=|\bar{G}|(2\gg-2)+\sum_{i=1}^k (|\bar{G}|-\ell_i).$$
Therefore, $\gg=0$, $k=3$ and $\ell_i=\thi |\bar{G}|$. Since $\bar{\alpha}$ has exactly three fixed points, the same holds for every nontrivial element in $\aut({\bar \cE})$ fixing a point. Thus, $\ell_i$ produces as many as $\textstyle\frac{2}{3}\ell_i=\textstyle\frac{2}{9}|\bar{G}|$ elements of order $3$. Since $k=3$, this shows that $\bar{G}\setminus \bar{H}$ has at least $\textstyle\frac{2}{3}|\bar{G}|$ elements of order $3$. On the other hand,  $|\bar{G}|-|\bar{H}|=\textstyle\frac{2}{3}|\bar{G}|$. From this Claim (iii) follows. Finally, Result \ref{resbl1} gives Claim (iv).
\end{proof}
\section{Case $d\neq p$}
\label{vegyes}
Let $G$ be  a subgroup of $\aut(\cX)$ whose order is $d^h$ where $d$ is a prime and $h\geq 1$ is an integer.

If the quotient curve $\bar{\cX}=\cX/G$ has genus $\gg(\bar{\cX})\ge 2$, then
the Hurwitz genus formula yields $|G|\leq \gg(\cX)-1$.

If $\bar{\cX}$ is elliptic, then the cover $\cX|\bar{\cX}$ is ramified, and the Hurwitz genus formula gives $2(\gg(\cX)-1)\geq |G|-|\ell|$ where
where $\ell$ is any short orbit of $G$. Since $|\ell|\leq |G|/d$, this yields $|G|\le \frac{2d}{d-1}(\gg(\cX)-1)$.

We are left with the case where $\bar{\cX}=\cX/G$ is rational.
Assume that $|G|>4(\gg(\cX)-1).$ Then the quotient curve $\bar{\cX}$ is rational; see, for instance, the proof of \cite[Theorem 11.56]{HKT}. Therefore, the Hurwitz genus formula applied to $G$ gives
\begin{equation} \label{eqZbound}
2(\gg(\cX)-1)=  -2d^h+ \sum_{j=1}^k (d^h-\ell_j)=-\ell_1-\ell_2+\sum_{j=3}^k (d^h-\ell_j)
\end{equation}
where $|G|=d^m$ and $\ell_1,\ldots \ell_k$ denote the short orbits of $G$ on $\cX$. Since $\gg(\cX) \geq 2$, (\ref{eqZbound}) yields that $G$ has at least $3$ short orbits. Also, in \eqref{eqZbound} each $\ell_j$ is a power of $d$. For $d=3$ this yields $|G|\leq 9(\gg(\cX)-1)$ and
if equality holds with $\gg(\cX)\neq 2$ then $G$ is not abelian by (ii) of Result \ref{res1}. Actually, the case $\gg(\cX)=2$ with $|G|=9$ cannot occur from Result \ref{res15bis}.

Dealing with the case  $d\geq 5$, assume that $|G| >\frac{2d}{d-3}(\gg(\cX)-1)>4(\gg(\cX)-1)$. Since $\ell_i$ divides $d^m$, $\ell_i \leq d^{m-1}$ holds for every $i$. If $G$ has exactly three short orbits then \eqref{eqZbound} gives,
$$2(\gg(\cX)-1)= -2d^m+ (3d^m-\ell_1-\ell_2-\ell_3) \geq d^m-3d^{m-1}=d^{m-1}(m-3),$$
whence $|G|=d^h \leq \frac{2d}{d-3}(\gg(\cX)-1)$, a contradiction. Similarly, if  $G$ has four short orbits, then \eqref{eqZbound} gives
$$2(\gg(\cX)-1)=2d^m-\ell_1-\ell_2-\ell_3-\ell_4 \geq 2d^m-4d^{m-1}=d^{m-1}(2d-4),$$
whence $|G|=d^h \leq \frac{2d}{2d-4}(\gg(\cX)-1)<\frac{2d}{d-3}(\gg(\cX)-1)$, again a contradiction.
We are left with the case where $G$ has at least $5$ short orbits. Since $d\geq 5$, the Hurwitz genus formula gives $$2(\gg(\cX)-1) \geq 3d^m-\ell_1-\ell_2-\ell_3-\ell_4-\ell_5 \geq 3d^m-5d^{m-1} \geq 3d^m-d^m=2d^m$$ whence $|G|=d^h \leq 2(\gg(\cX)-1)$; again a contradiction.

Therefore, the following result is obtained.
\begin{theorem}
\label{proCdic112017} If $G$ is an $d$-subgroup of $\aut(\cX)$ with an odd prime $d$ different from $p$
\label{eqC11dic2017}
\begin{equation} |G|\le
\begin{cases}
 {\mbox{$9(\gg(\cX)-1)$ for $d=3$}};\\
 {\mbox{$\frac{2d}{d-3}(\gg(\cX)-1)$ for $d>3$}}.
\end{cases}
\end{equation}
For $d=3$, if equality holds then $G$ is not abelian and $\gg(\cX)\neq 2$.
\end{theorem}
For abelian groups, the above bound is sharp for $d\geq 5$. In fact, the Fermat curve $\cF_d$ of affine equation $x^d+y^d+1=0$ has genus $\ha(d-1)(d-2)$ and $\aut(\cF_q)$ has an abelian subgroup of order $d^2$ which is the direct product of two cyclic groups of order $d$.

For non-abelian groups a slight improvement on the bound in Proposition \ref{proCdic112017} is given.
\begin{theorem}
\label{lem13jan2018} Let $G$ be a non-abelian $d$-subgroup of $\aut(\cX)$ where $d$ is an odd prime different from $p$. If $Z$ is an order $d$ subgroup of $Z(G)$ such that the quotient curve $\bar{\cX}=\cX/Z$ has genus at most $1$ then $\bar{\cX}$ is elliptic and
\begin{equation}
\label{eqA3jan2018}
|G|\le \textstyle\frac{2d}{d-1}(\gg(\cX)-1),
\end{equation}
apart from the case where
\begin{equation}
\label{eq27luglio2018}
{\mbox{$d=3$, and $|G|=9(\gg(\cX)-1)$.}}
\end{equation}
\end{theorem}
\begin{proof} Obviously, $Z$ is a proper subgroup of $G$. If the quotient curve $\bar{\cX}=\cX/Z$ is rational then the quotient group $\bar{G}=G/Z$ is a isomorphic to a subgroup of $PGL(2,\mathbb{K})$. From Result \ref{res5}, $\bar{G}$ is a cyclic group. But this yields that $G$ is an abelian group, a contradiction.

Therefore $\bar{\cX}$ is elliptic. Since the set consisting of all fixed points of $Z$ is left invariant under the action of $G$, it is partitioned in $G$-orbits, say $\Omega_1,\ldots,\Omega_k$.
From the Hurwitz genus formula applied to $Z$,
$$2(\gg(\cX)-1)= \sum_{i=1}^k |\Omega_i|(|Z_{P_i}|-1)=\sum_{i=1}^k \frac{|G|}{|G_{P_i}|}(|Z_{P_i}|-1)=\sum_{i=1}^k |G| \frac{|Z_{P_i}|}{|G_{P_i}|}\frac{1}{|Z_{P_i}|}
(|Z_{P_i}|-1) \geq \frac{e-1}{e}\,|G|\sum_{i=1}^k \frac{|Z_{P_i}|}{|G_{P_i}|}.$$
We show that $G_{P_i}=Z_{P_i}$ with a possible exception for $d=3$ and $|G_{P_i}|=3|Z_{P_i}|$. Let $\bar{P}_{i}$ be the point of $\bar{\cX}$ lying under the $Z$-orbit $\Omega_i$. Then $\aut(\bar{\cX})$ has a subgroup $\bar{M}_i$ isomorphic to $G_{P_i}/Z_{P_i}$ that fixes $\bar{P}_i$. From (ii) of Result \ref{res4}, this subgroup of order $d$ is actually trivial with just one possible exception when $d=3$ and   $[G_{P_i}:Z_{P_i}]=3$. For $d>3$, the above bound reads $$2(\gg(\cX)-1)=k \frac{e-1}{e}|G|\geq \frac{d-1}{d}|G|$$ whence (\ref{eqA3jan2018}) follows. In the exceptional case, (\ref{eq27luglio2018}) holds with $\gg(\cX)>2$.
\end{proof}

\section{Extremal $3$-Zomorrodian curves}
\label{extremal12agosto}
From now on $$p\neq 3,$$
is assumed. We use the name of ``extremal $3$-Zomorrodian curve'' for any curve $\cX$ with a non-abelian $3$-subgroup $G$ in $\aut(\cX)$ attaining the bound in Proposition \ref{lem13jan2018} for $d=3$. Obviously, the group $G$ with respect to an extremal $3$-Zomorrodian curve $\cX$ is a Sylow $3$-subgroup of $\aut(\cX)$.
We begin with two sporadic examples.
\begin{example}
\label{ex27Aluglio} {\em{The genus $10$ curve $\cX$ with affine equation $x^6y^3+x^3y^6+1=0$ has a non-abelian automorphism group of order $81$ isomorphic to ``SmallGroup''(81,9) in the  MAGMA database.}}
\end{example}
\begin{example}
\label{ex27luglio} {\em{The Fermat curve $\cF_9$ of degree $9$ has genus $28$ and $\aut(\cF_9)$ has a non-abelian subgroup of order $243$ which is the semidirect product of an abelian group of order $81$ (direct product of two cyclic groups) and a subgroup of order $3$.}}
\end{example}
We remark that  Example  \ref{ex27Aluglio} is the quotient curve of $\cF_9$ with respect to the center of its automorphism group of order $243$. This shows that the non-elliptic case in Proposition \ref{prop29luglio2018} can occur.

The following proposition shows that no extremal  $3$-Zomorrodian curve of genus $4$ exists.
\begin{lemma}
\label{nog=4}
There exists no extremal $3$-Zomorrodian curve of genus $\gg(\cX)=4$.
\end{lemma}
\begin{proof}
By contradiction, $\cX$ is a genus $4$ curve and $G$ is a subgroup of order $27$ in $\aut(\cX)$.
From (ii) of Result \ref{res1}, $G$ is not abelian as $27>4\gg(\cX)+4=20$.

Up to isomorphisms, there exist two non-abelian groups of order $27$, namely $G=U(3,3)=(C_3\times C_3)\rtimes C_3$  and $G= C_9 \rtimes C_3$. In both cases, the center $Z(G)$ of $G$ has order $3$, and the $G/Z(G)$ is an elementary abelian group of order $9$. In the latter case, $G=\langle a,b:a^9=b^3=1,b^{-1}ab=a^4\rangle$ and $G$ has three subgroups of order $9$ while the elements of order $3$ form an elementary abelian subgroup $H$ of order $9$.

The Hurwitz genus formula applied to $Z(G)$ reads, $$6=2\gg(\cX)-2=6(\gg(\tilde{\cX})-1) +2\Delta,$$
where  $\Delta$ counts the number of fixed points of $Z(G)$. Hence $\gg(\tilde{\cX}) \leq 2$.

If $\tilde{\cX}$ is rational, then a Sylow $3$-subgroup of $\aut(\tilde{\cX})$ is cyclic by Result \ref{res5}, a contradiction. Case $\gg(\tilde{\cX})=2$ does not occur by Result \ref{res15bis}.

Hence $\tilde{\cX}$ is elliptic, and $\Delta=3$. The latter claim yields that $G$ has a short orbit of size $3$. Therefore, $|G_P|=9$ for any point $P$ fixed by $Z(G)$. From \cite[Theorem 11.49]{HKT} $G_P$ is a cyclic group of order $9$, this rules out case $G \cong U(3,3)$. Finally, to deal with case $G=C_9\rtimes C_3$ two cases are distinguished according as an element $g\in G\setminus Z(G)$ of order $3$ has a fixed point or not. In the former case, no element in $G\setminus Z/G)$ has a fixed point, as being conjugate to $g$ or $g^2$ in $G$. Furthermore, no element $G$ of order $9$ other than those in $G_P$ has a fixed point as its cube is in $Z(G)$. Therefore, from the Hurwitz genus formula applied to $G$,  $6=2\gg(\cX)-2=54(\gg(\hat{\cX})-1)+24$, which is a contradiction. In the latter case,
$g$ has exactly three fixed points. In fact, if $\tilde{H}$ is the quotient group $H/Z(G)$ then the Hurwitz genus formula applied to $\tilde{H}$  gives
$0=2\gg(\tilde{\cX})-2=3(\gg(\hat{\cX})-2)+2\lambda$
where $\hat{\cX}$ is the quotient curve $\tilde{\cX}/\tilde{H}$ and $\tilde{\lambda}$ counts the number of fixed points of $\tilde{H}$ on $\tilde{\cX}$. From this, $\gg(\hat{\cX})=0$ and $\tilde{\lambda}=3$. Hence the unique short $H$-orbit has size $9$.

Therefore, $G$ has two short orbits, one consists of the three fixed points of $Z(G)$, the other has length $9$. From the Hurwitz genus formula applied to $G$,
$6=2\gg(\cX)-2=-54+24+18=-12$, which is a contradiction.
\end{proof}
The following result shows that the center of a Sylow $3$-subgroup of $\aut(\cX)$  plays an important role in the study of extremal $3$-Zomorrodian curves $\cX$.
\begin{proposition}
\label{prop29luglio2018} Let $G$ be a Sylow $3$-subgroup of $\aut(\cX)$ of an extremal $3$-Zomorrodian curve $\cX$ of genus $\gg(\cX)=3^h+1$ with $h\geq 2$.
Let $Z$ be an order $3$ subgroup of $Z(G)$.  Then the quotient curve $\bar{\cX}=\cX/Z$ is either elliptic (with zero $j$-invariant), or an extremal $3$-Zomorrodian curve of genus $3^{h-1}+1$ and $\bar{G}=G/Z$ is a Sylow $3$-subgroup of $\aut(\bar{\cX})$. Furthermore, there exists at most one subgroup $Z$ for which  $\bar{\cX}$ is elliptic. If such a subgroup $Z$ exists and $h\geq 3$ then either $Z(G)\cong C_3$, or $Z(G)\cong C_3\times C_3$.
\end{proposition}
\begin{proof} By Proposition \ref{lem13jan2018}, $\gg(\bar{\cX})>0$. From Lemma \ref{nog=4}, $|G|\geq 81$ and hence $\gg(\bar{\cX})\neq 2$ by Result \ref{res15bis}.

 First case $\gg(\bar{\cX})>2$ is considered. From the Hurwitz genus formula applied to $Z$,  $\gg(\cX)-1 \ge 3(\gg(\bar{\cX})-1)$. Since $\gg(\cX)-1=9|G|$ and $|G|=3|\bar{G}|$,  this yields $\gg(\bar{\cX})-1\geq 9$. Now, from Proposition \ref{lem13jan2018} applied to $e=3$, $\gg(\bar{\cX)}-1=9|\bar{G}|$ follows.

From now on, $\gg(\bar{\cX})=1$. With notation as in the proof of Proposition \ref{lem13jan2018},  $|G_{P_i}|=3|Z_{P_i}|$, and
$$2(\gg(\cX)-1)=k \textstyle\frac{2}{9}|G|.$$ For $k\geq 2$, this yields (\ref{eqA3jan2018}) with equality. Therefore, $k=1$, and the set $\Omega$ of all fixed points of $Z$ has size $\frac{1}{9}|G|$. Take a point $P\in \cX$ such that $G_P$ is non-trivial. Let
$\bar{P}$ be the point of $\bar{\cX}$ lying under $P$ in the cover $\cX|\bar{\cX}$. There is a subgroup $\bar{M}$ in $\aut(\bar{\cX})$ which acts on $\bar{\cX}$ as $G_P$ does on the $Z$-orbits. The quotient curve $\hat{\cX}=\bar{\cX}/\bar{M}$ is rational as $\bar{M}$ fixes $\bar{P}$. More precisely, the Hurwitz genus formula applied to $\bar{M}$ yields
$0=-6+2 \lambda$ where $\lambda$ counts the fixed points of $\bar{M}$. Hence, $\lambda=3$. Observe that $\bar{M}$ is not contained in the center $Z(\bar{G})$, otherwise the set of fixed points of $\bar{M}$ is left invariant by $\bar{G}$ and hence $|\bar{G}|=9$ which contradicts $|G|\geq 81$. This yields that $|Z(\bar{G})|=3$, since the group $Z(\bar{G})\bar{M}$ generated by $Z(\bar{G})$ and $\bar{M}$ has order at most $9$ but $\bar{M}\lvertneqq Z(\bar{G})$. Therefore $|Z(G)|=3,9$.

From the Hurwitz genus formula applied to $G$,
$$\textstyle\frac{2}{9}|G|=2(\gg(\cX)-1)=-2|G|+(|G|-\textstyle\frac{1}{9}|G|)+\sum_{i=1}^j (|G|-\ell_j)$$
where $\ell_1,\ldots,\ell_j$ are the sizes of the short $G$-orbits other than $\Omega$. Since $|G|$ and $\ell_i$ are powers of $3$, this yields $j=2$ and $\ell_1=\ell_2=\frac{1}{3}|G|$. Therefore,
$|G_P|=9$ for $P\in\Omega$, otherwise $|G_P|=3$ or $G_P$ is trivial. From this we infer that $Z$ is the unique order $3$ subgroup of $Z(G)$ such that the quotient curve $\bar{\cX}$ is elliptic. In fact, if $Z(G)$ contains a subgroup $U$ of order $3$ other than $Z$ such that the quotient curve $\cX/U$ is elliptic, then the above argument applied to $U$ shows that the set of fixed points of $U$ has size $\frac{1}{9}|G|$ and hence it must coincide with $\Omega$. But this is impossible in our case, as the stabilizer $G_P$ with $P\in \Omega$ is a cyclic group, and hence it can contain only one subgroup of order $3$.


Let $\theta$ be one of the two short $G$-orbits of length $\textstyle\frac{1}{3}|G|$, and take point $Q\in \theta$. Then the stabilizer $U=G_Q$ has order three, and $U\cap Z=\{1\}$. Furthermore, the quotient group $\bar{U}=UZ/Z$ is isomorphic to $U$ and it is a subgroup of $\aut(\bar \cX)$ of order three. The same holds true for the other short $G$-orbit $\sigma$.

The quotient curve $\hat \cX= \bar \cX/\bar U$ is rational. In fact, the cover $\bar \cX \rightarrow \hat \cX$ ramifies as $U$ preserves the $Z$-orbit $\theta$ containing the point $Q$.
From the Hurwitz genus formula applied to $\bar{U}$,
$$0=2\gg(\bar \cX)-2=3(2 \gg(\hat\cX)-2)+2 \bar{\lambda},$$
where $\bar{\lambda}$ is the number of fixed points of $\bar U$. Hence $\bar{\lambda}=3$. If $\bar{Q}_1,\bar{Q}_2,\bar{Q}_3$ are the fixed points of $\bar{U}$ then each of the three $Z$-orbits $\theta_1,\theta_2,\theta_3$ lying over $\bar{Q}_1,\bar{Q}_2,\bar{Q}_3$ respectively is preserved by $U$. If a point in such a $Z$-orbit is fixed by $U$ then the same $Z$-orbit is fixed by $U$ pointwise. Since one of these three $Z$-orbits contains $Q$, it turns out that $U$ fixes either $3$, $6$ or $9$ points in $\theta_1\cup \theta_2\cup \theta_3$. Furthermore, $U$ has no other fixed point. From now on, $|Z(G)|=9$ is assumed.

We are in a position to prove that $V=Z(G)$ is not a cyclic group of order $9$. Assume on the contrary that $V\cong C_9$.

Let $\tilde{\cX}=\cX/U$. From the Hurwitz genus formula applied to $U$,
\begin{equation}
\label{eq18apr2018}
2\gg(\cX)-2=3(2\gg(\tilde{\cX})-2)+2\tilde{\lambda}
\end{equation}
with $\tilde{\lambda} \in \{3,6,9\}$. Since $\gg(\cX)-1$ is divisible by $27$, (\ref{eq18apr2018}) and assumption $\gg(\cX)\geq 28$ yield that $\gg(\tilde{\cX})-1$ is not divisible by $9$.
Since no nontrivial element in $V$ fixes a point off $\Omega$, each $V$-orbit disjoint from $\Omega$ has size $9$. Since the number of fixed points of $U$ does not exceed $9$, this yields that $U$ has as many as $9$ fixed points.

We find all points in $\cX$ where the cover $\cX|\tilde{\cX}$ ramifies.
If some nontrivial element $v\in V$ fixes a point $P\in \cX$, then $P$ is also fixed by $Z$ since either $v\in Z$, or $v\not\in Z$ and $v^3\in Z$. This shows that $\Omega$ is the set of all points which are fixed by some nontrivial element in $V$. More precisely, if $P\in \Omega$ then either $|V_P|=3$, or $|V_P|=9$. Since $\Omega$ is a $G$-orbit and $V$ is a normal subgroup, either $|V_P|=3$ for all $P\in \Omega$, or $|V_P|=9$ for all $P\in \Omega$.
As $Z\neq U$, we have $U \cap V=\{1\}$. Then $\tilde{V}=VU/U$ is a subgroup of $\aut(\tilde{\cX})$ isomorphic to $V$.
Let $\tilde{\Omega}$ denote the set of all points of $\tilde{\cX}$ lying under the points of $\Omega$ in the cover
$\cX|\tilde{\cX}$. Since $U$ has no fixed point in $\Omega$, we have $|\tilde{\Omega}|=\frac{1}{27}|G|$.

This shows that either $|\tilde{V}_{\tilde{P}}|=3$ for all points $\tilde{P}\in \tilde{\Omega}$, or $|\tilde{V}_{\tilde{P}}|=9$ for all points $\tilde{P}\in \tilde{\Omega}$. Now we prove that no nontrivial element in $\tilde{V}$ has a fixed point outside $\tilde{\Omega}$.

By way a contradiction, take a point $\tilde{R}$ fixed by a nontrivial element $\tilde{v}\in \tilde{V}$. Let $\rho$ be the $U$-orbit consisting of all points lying over $\tilde{R}$ in the cover $\cX|\tilde{\cX}$. Then $\rho$ is left invariant not only by $U$ but also by a nontrivial subgroup $W$ of $V$ fixing a point in $\rho$ where $WU/U\cong \langle \tilde{v}\rangle$.
Since $U\cap W=\{1\}$, the subgroup $T=UW$ has order either $27$ or $9$, according as $|W|=9$ or $|W|=3$. The former case cannot actually occur. In fact, since $|\rho|=3$ we have $|T_R|=\frac{1}{3}|T|$ for any $R\in \rho$. On the other hand $|T_R|=3$ by $R\not\in \Omega$. Therefore, $T=U\times Z$, and $\rho$ is a common orbit of $U$ and $Z$. The latter claim shows that the point $\bar{R}$ lying under $\rho$ in the cover $\cX|\bar{\cX}$ is fixed  by $\bar{U}$. Since
no fixed point of $U$ lies in $\rho$, this shows that $\bar{U}$ fixes at least four points, namely $\bar{R}$ and each of the three points lying under the fixed points of $U$ in the cover $\cX|\bar{\cX}$. But this contradicts the previous claim $\bar{\lambda}=3$.

From the Hurwitz genus formula applied to $\tilde{V}$, $$\gg(\tilde \cX)-1=9(\gg(\cX^*)-1)+c\frac{|G|}{27},$$
where $\cX^*$ denotes the quotient curve $\tilde \cX/\bar V$ and either $c=2$, or $c=8$, according as $|V_P|=3$ or $|V_P|=9$ for $P\in \Omega$.
A contradiction is now obtained since $9$ divides both $9(\gg(\cX^*)-1)$ and $ \frac{|G|}{27}$, but it does not  $\gg(\tilde \cX)-1$.
\end{proof}
\begin{rem}
\label{lem26sep} {\em{The proof of Proposition \ref{prop29luglio2018} also shows some useful properties of the action of a Sylow $3$-subgroup $G$ of $\aut(\cX)$ of an extremal $3$-Zomorrodian curve $\cX$ of genus $\gg(\cX)=3^h+1$ with $h\geq 2$. Here, we mention some of them: $G$ has exactly $3$ short orbits, namely $\Omega$ of size $\frac{1}{9}|G|$, $\theta$ and $\sigma$ both of size $\frac{1}{3}|G|$. The subgroup $Z<Z(G)$ such that $\cX/Z \cong \bar{\mathcal{E}}$ is elliptic, fixes $\Omega$ pointwise. If $|Z(G)|=9$, no other non-trivial element in $Z(G)\cong C_3 \times C_3$ has fixed a point on $\cX$. The stabilizer of $P\in \Omega$ fixes exactly three points in $\Omega$ while the stabilizer of $Q\in \theta$ fixes as many as $|Z(G)|$ points, each lying in $\theta$, and the same holds for $Q\in \sigma$.}}
\end{rem}

\section{Elliptic type extremal $3$-Zomorrodian curves}
\label{ge}
In the light of Proposition \ref{prop29luglio2018}, it is useful to adopt the term of ``elliptic type'' for an extremal $3$-Zomorrodian curve $\cX$ if the center $Z(G)$ of a Sylow $3$-subgroup of  $\aut(\cX)$ contains a subgroup $Z$ of order $3$ such that the quotient curve $\cX/Z$ is elliptic. With this definition, all examples from Section \ref{ss29luglioA} are of elliptic type. By Lemma \ref{nog=4}, any extremal $3$-Zomorrodian curve of genus $10$, in particular Example \ref{ex27Aluglio}, is of elliptic type. Instead, Example \ref{ex27luglio} is not of elliptic type.

First we collect some basic results on the abstract structure of the Sylow $3$-subgroups of an extremal $3$-Zomorrodian curves of elliptic type.
\begin{lemma}
\label{lemC29luglio2018} Let $G$ be a Sylow $3$-subgroup of an extremal $3$-Zomorrodian curve of elliptic type. Then
\begin{itemize}
\item[(i)] $G$ can be generated by two elements;
\item[(ii)]  $[G:G']=9$; $\Phi(G)=G'$;
\item[(iii)]  $G$ contains exactly four maximal subgroups, each is normal and of index $3$.
\end{itemize}
\end{lemma}
\begin{proof} Let $\hat{\cX}=\cX/G'$. From Result \ref{resgroup}, the factor group $\hat{G}=G/G'$ is an abelian subgroup of $\aut(\hat{\cX)}$. If $\hat{\cX}$ were rational then by Result \ref{res15bis}  $\hat{G}$ would be cyclic, hence $G$ itself would be cyclic by (iii) of Result \ref{resgroup5}, a contradiction with Proposition \ref{proCdic112017}. Similarly, if $\gg(\hat{\cX})$ were bigger than $2$, then $\hat{\cX}$ would be an extremal $3$-Zomorrodian curve which is impossible as $\hat{G}$ is abelian. We are left with the elliptic case, that is, $\bar{\cX}=\bar \cE$. Let $P\in\cX$ be a fixed point of some non-trivial element of $G$. Let $\bar{P}\in \bar \cE$ be the point lying under $P$ in the cover $\cX| \bar \cE$. Then $\bar{G}_{\bar{P}}$ is non-trivial. Since $\bar{G}$ is abelian,
this yields $|\bar{G}|=9$ by Lemma \ref{lem30luglio2018}. This shows first claim in (ii). Since $G$ is not cyclic, the Burnside basis theorem yields $[G:\Phi(G)]\geq 9$. This together with $G'\le \Phi(G)$ and $[G:G']=9$ give the second claim in (ii). Since $\Phi(G)$ contains all maximal subgroups of $G$, claim (ii) yields that $G/\Phi(G)$ has exactly four subgroups, each being normal and of order $3$. 
\end{proof}
 Proposition \ref{prop29luglio2018} and Lemmas \ref{nog=4}, \ref{lemC29luglio2018} have the following corollary.
\begin{lemma}
\label{lem29luglioA} Let $\cX$ be an extremal $3$-Zomorrodian of genus $10$. Then a Sylow $3$-subgroup $G$ of $\aut(\cX)$ of maximal nilpotency class. More precisely, the possibilities for $G$ are
``SmallGroup''(81,7), ``SmallGroup''(81,8), ``SmallGroup''(81,9), and  `SmallGroup''(81,10).
\end{lemma}
\begin{proof} By Lemma \ref{nog=4} and Proposition \ref{prop29luglio2018}, the quotient curve $\cX/Z$ is elliptic for any order $3$ subgroup of $Z(G)$.  Also, from Proposition \ref{prop29luglio2018}, $Z(G)$ has a unique subgroup of order $3$. Among the $15$ groups of order $81$ those satisfying (ii) of Lemma \ref{lemC29luglio2018} have center of order $3$, and they are isomorphic to one of the four ``SmallGroup'''s  listed above.
\end{proof}

We prove a key result on the abstract structure of a Sylow $3$-subgroup $G$  of an extremal $3$-Zomorrodian curve of elliptic type.  From Result \ref{resgroup3} and Lemma \ref{lemB29luglio2018}, we already know (for $|G|\geq 3^6$), the existence of a unique abelian or minimal non-abelian subgroup of index $3$. But we need some more properties of such a subgroup in order to determine the abstract structure of $G$.
\begin{lemma}
\label{keylemma} Let $G$ be a Sylow $3$-subgroup of an extremal $3$-Zomorrodian curve of elliptic type. Then exactly one of the subgroups of $G$ of index $3$ is either abelian or minimal non-abelian. Let $H$ be such a subgroup. If $H$ is minimal non-abelian then $H'\le Z(G)$ and the quotient curve $\cX/H'$ is elliptic.
\end{lemma}
\begin{proof} Choose a point $P\in \cX$ fixed by some non-trivial element of $G$. Let $Z$ be a subgroup of $Z(G)$ for which the quotient curve $\bar{\cE}=\cX/Z$ is elliptic. Then the factor group $\bar{G}=G/Z$ is a subgroup of $\aut(\bar{\cE})$ of order $\thi |G|$ such that $\bar{G}_{P}$ is non-trivial for the point $\bar{P}$ lying under $P$ in the cover $\cX|\bar{\cE}$.
From Lemma \ref{lem28jun}, $\bar{G}=\bar{H}\rtimes \bar{G}_{\bar{P}}$ where $\bar{H}=\bar{G}\cap J(\bar{\cE})$ and $\bar{H}=\bar{U}\times \bar{V}$ with two nontrivial cyclic groups $\bar{U}$ and $\bar{V}$. Let $H$ be the subgroup of $G$ containing $Z$ such that $H/Z=\bar{H}$. Then $H$ is a subgroup of $G$ of index $3$. To show that $H$ satisfies the property required, we may assume $H$ to be non-abelian.

First we prove that $d(H)=2$, that is, $H$ can be generated by two elements. Let $U$ and $V$ be the (normal) subgroups of $G$ containing $Z$ for which $\bar{U}=U/Z$ and $\bar{V}=V/Z$. Since $\bar{U}$ is cyclic, we have that $U$ is abelian. More precisely, either $U$ is cyclic, or $U=U_1\times Z$ where $U_1$ is a cyclic group isomorphic to $\bar{U}$. The same holds for $V$. Every element $w\in H$ is a product $w=uv$ with $u\in U, v\in V$. Obviously, if both $U$ and $V$ are cyclic, then $d(H)=2$. Assume that $U=U_1\times Z$ and $V$ is cyclic. Then $w=u_1zv$ with $u_1\in U_1, z \in Z, v\in V$. Since $z\in V$, this yields $w=u_1v'$ with $u_1\in U_1,v'\in V$. Therefore, $d(H)=2$. If both $U$ and $V$ are not cyclic, but $z\in \langle U_1,V_1 \rangle$ then again $w\in\langle U_1,V_1 \rangle$ and $d(H)=2$. Otherwise, $z\not\in \langle U_1,V_1 \rangle$. We show that the set $U_1V_1=\{u_1\in U_1,v_1\in V_1\}$ is subgroup of $G$. From $\bar{U}\bar{V}=\bar{V}\bar{U}$, we have
$\bar{u}_1\bar{v}_1\bar{u}_2\bar{v}_2=\bar{u}_1\bar{u}_2\bar{v}_1\bar{v}_2,$
whence $u_1v_1u_2v_2=u_1u_2v_1v_2z^i$ with $0\leq i \leq 2$. Since $z\not\in \langle U_1,V_1 \rangle$ is assumed, $i$ must be equal to zero. Therefore, $U_1V_1$ is indeed a subgroup of $G$. Since
$|U_1V_1|=|U_1||V_1|=|\bar{U}||\bar{V}|=|\bar{H}|=\thi |H|$, the subgroup $U_1V_1$ of $H$ has index $3$, and hence $H=U_1V_1\times Z$. Therefore $U_1V_1\cong \bar{H}$. This implies that $U_1V_1$ is an abelian group. But then $H$ itself must be abelian, a contradiction. Thus $d(H)=2$.

Next we prove that $H'=Z$. As $Z\le H$ and the factor group $H/Z$ is abelian, Result \ref{resgroup} yields $H'\le Z$. On the other hand, $H'$ is non-trivial as $H$ is assumed to be non-abelian. Therefore, $H'=Z$. Now,  Result \ref{resgroup1} shows that $H$ is minimal non-abelian, and the quotient curve  $\cX/H'=\bar{\cX}$ is elliptic.

Finally, uniqueness of $H$ follows from Result \ref{resgroup3} when $H$ is non-abelian, and from Results \ref{resgroup5}, \ref{qu14sep} and Lemma \ref{lemC29luglio2018}.
\end{proof}
\begin{lemma}
\label{20sett2018} Let $H$ be as in Lemma \ref{keylemma}. If $|Z(G)|=9$ and $|G|\geq 3^6$ then $G$ has as many as $\frac{4}{9}|G|+26$ elements of order $3$.
\end{lemma}
\begin{proof}  We keep our notation from Remark \ref{lem26sep}.

We prove first that $G \setminus H$ contains at least as many as $\frac{2}{9}|G|$ elements of order $9$. Let $P\in \Omega$, and take a generator $g$ of $G_P$. From Remark \ref{lem26sep}, $g$ fixes exactly three points of $\cX$, each lying in $\Omega$. Furthermore, $g\in G\setminus H$, as its image $\bar{g}$ in the natural homomorphism $G\mapsto \bar{G}$ has a fixed point on $\bar{\cE}$ and hence $\bar{g}\not\in \bar{H}$, equivalently $g\not\in H$. Thus, $G$ contains exactly $\frac{1}{3}|\Omega|=\frac{1}{27}|G|$ pairwise distinct cyclic subgroups of order $9$ each with a fixed point in $\Omega$. Since each such subgroup contains exactly $6$ elements of order $9$, $G \setminus H$ contains indeed at least $\frac{6}{27}|G|=\frac{2}{9}|G|$ elements of order $9$.

We prove next that $G \setminus H$ contains at least $\frac{4}{9}|G|$ elements of order $3$.
From Remark \ref{lem26sep}, no non-trivial element in $Z(G) \setminus Z$ has a fixed point in $\cX$. Hence, if $\cX/Z$ is elliptic then $\cX/Z(G)$ is also elliptic. Let $Q \in \theta$ and $G_Q=\langle u \rangle$. Then $u$ has order $3$. Furthermore, by Remark \ref{lem26sep}, $u$ has as many as $9$ fixed points in $\cX$, each lying in $\theta$. Therefore, these $9$ points form a $Z(G)$-orbit as $|Z(G)|=9$. Thus $\theta$ splits into $\frac{1}{27}|G|$ $Z(G)$-orbits each being the set of fixed points of $G_Q$ where $Q$ ranges over $\theta$. Since $|G_Q|=3$, each $Z(G)$-orbit in $\theta$ is preserved exactly by $S_Q=G_Q\times Z(G)$. In particular, $S_Q$ contains $9$ elements from $Z(G)$ and $18$ elements off $Z(G)$. All these facts remain true for $Q\in\sigma$.

Our aim is to show that among the non-central elements of $G$ of order $3$, exactly $\frac{4}{9}|G|$ have a fixed point in $\theta\cup\sigma$. For this purpose, look at the intersection of two such subgroups, namely $S_Q$ and $S_R$ with $Q,R\in \theta\cup\sigma.$ From  $|S_Q|=|S_R|=27$ and $|Z(G)|=9$, either $S_Q \cap S_R=Z(G)$ or $S_Q=S_R$.
The latter case occurs if and only if $G_R<S_Q$, equivalently, $G_R$ preserves the $Z(G)$-orbit containing $Q$. This shows that the coincidences $S_Q=S_R$ can be counted by computing the number of $Z(G)$-orbits lying in $\theta \cup \sigma$ on which $G_R$ acts. This computation can be carried out on the quotient curve $\hat{\cX}=\cX/Z(G)$, since the $Z(G)$-orbits preserved by $G_R=\langle v \rangle$ are as many as the fixed points of $\hat{v}$ in $\hat{\cX}$ where $\hat{v}$ is the image of $v$ in the natural homomorphism $G\mapsto \hat{G}$. Observe that $\hat{\cX}$ is elliptic as  the elements of $Z(G)$ with a fixed a point in $\cX$ are exactly those in $Z$. Further, as $v$ fixes $R$,
$\hat{v}$ fixes the point $\hat{R}\in \hat{\cX}$ lying under $R$ in the cover $\cX|\hat{\cX}$. There exist exactly two more fixed points of $\bar{v}$, say $\hat{S}_1,\hat{S}_2$. Let $S_1\in \cX$ be a point lying over $\hat{S}_1$ in the cover $\cX|\hat{\cX}$. Then $G_R\times Z(G)$ preserves the $Z(G)$-orbit $\Delta_1$ containing $S_1$. Therefore, for some $z\in Z(G)$, $vz$ fixes $S_1$ whence $S_1\in \Omega \cup \theta \cup \sigma$ follows. Actually, $S_1\not\in \Omega$, as $Z$ is the unique subgroup of order $3$ in $G_P$ whereas $vz\in Z$ would yield $v\in Z(G)$. Similarly, we have $S_2\in \sigma$. Since $\theta$ (and $\sigma$) is partitioned in $Z(G)$-orbits whose size is divisible by $9$, the number of $G_P$-invariant $Z(G)$ lie in $\theta$ (and $\sigma$) is divisible by $3$. It turns out that each of the three $Z(G)$-orbits $\{\Delta_1,\Delta_2,\Delta_3\}$ preserved by $G_R$ lie in either $\theta$ or $\sigma$. Such a triple $\{\Delta_1,\Delta_2,\Delta_3\}$ does not change when $R$ is chosen from $\Delta_2$ or $\Delta_3$.
Since $|\theta|=|\sigma|=\frac{1}{3}|G|$ and $G_R$ with $R\in \theta \cup\sigma$ has as many as $18$ elements off $Z(G)$, the number of elements of order $3$ in $G\setminus H$ is at least $$18 \cdot \textstyle\frac{2}{3}|G|\cdot \textstyle\frac{1}{9} \cdot \textstyle\frac{1}{3}=\textstyle\frac{4}{9}|G|.$$

Since $|G \setminus H|=\frac{2}{3}|G|=\frac{4}{9}|G|+\frac{2}{9}|G|$, the above two assertions we have already proven yield that $G \setminus H$ contains exactly $\frac{4}{9}|G|$ elements of order $3$. Since $H$ has exactly $26$ elements of order $3$ which are contained in $\Phi(G)$, the claim follows.
\end{proof}
\begin{theorem}
\label{the25agosto} Let $G$ be a Sylow $3$-subgroup of an elliptic type extremal $3$-Zomorrodian curve $\cX$ of genus $\gg(\cX)\ge 82$.  If $|Z(G)|=9$ then, in terms of generators and relations, $G$ is the group given in Result \ref{qu24agosto}.
\end{theorem}
\begin{proof} $G$ is not of maximal nilpotency class, and none of its subgroups  of index $3$ is abelian by Result  \ref{qu24agosto}. Thus, the claim  follows from Lemmas \ref{keylemma}, \ref{lemB29luglio2018} and Result  \ref{qu24agosto}.
\end{proof}
\begin{theorem}
\label{the25agostoA}  Let $G$ be a Sylow $3$-subgroup of an elliptic type extremal $3$-Zomorrodian curve $\cX$ of genus $\gg(\cX)\ge 82$.  If $|Z(G)|=3$ then, in terms of generators and relations, $G$ is the group given in Result \ref{qu14sep}.
\begin{proof} From Lemma \ref{keylemma}, $G$ has a unique subgroup $H$ of index $3$ which is either abelian or minimal non-abelian. Definition and properties of $H$ are described in Lemmas \ref{keylemma}, \ref{20sett2018} and in their proofs.

Two cases are treated separately according as $H$ is abelian or not.
\subsubsection*{Case of abelian $H$}  Result \ref{resgroup1} together with (ii) of Lemma \ref{lemC29luglio2018} show that $G$ has maximal nilpotency class. Also, from the initial part of proof of Result \ref{resbl1}, $H$ coincides with the fundamental group $G_1$ of $G$, and $G$ is one of the groups listed in (I) and (II) in the proof of Result \ref{resbl1}. Moreover, Result \ref{resbl1} applied to  $\bar G=G/Z(G)$ shows that
\begin{equation}
\label{casesbarg}
\bar{G}=
\begin{cases}
\langle \bar{s}_1,\bar{s}_2,\bar{\beta}|\bar{s}_1^{3^e}=\bar{s}_2^{3^{e}}=\bar{\beta}^3=1,[\bar{s}_1,\bar{\beta}]=\bar{s}_2,[\bar{s}_2,\bar{\beta}]=\bar{s}_2^{-3} \bar{s}_1^{-3},[\bar{s}_1,\bar{s}_2]=1\rangle;{\mbox{ for $|\bar{G}|=3^{2e+1}$}};\\
\langle \bar{s}_1,\bar{s}_2,\bar{\beta}|\bar{s}_1^{3^e}=\bar{s}_2^{3^{e-1}}=\bar{\beta}^3=1,[\bar{s}_1,\bar{\beta}]=\bar{s}_2,[\bar{s}_2,\bar{\beta}]=\bar{s}_2^{-3}\bar{s}_1^{-3},[\bar{s}_1,\bar{s}_2]=1\rangle;{\mbox{ for $|\bar{G}|=3^{2e}$}}.
\end{cases}
\end{equation}
Since some element of order $9$ falls in $G\setminus H$, for instance the generator of the stabilizer $G_P$ with $P\in \Omega$, Cases A(1) and B(1) do not occur. Also, $G\setminus H$ contains some element of order $3$, for instance the stabilizer $G_Q$ with $Q\in \theta$,  Cases A(2) and B(2) do not occur, as well. We are left with one of Cases A(3) and B(3).

In the former case $|G|=2^{2e+1}$, and we show that $Z(G)=\langle s_1^{3^{e-1}} \rangle$. If $G$ is given by A(3) then $[s_1^{3^{e-1}},\alpha]=[s_1^{3^{e-1}},s_2]=1$ by $s_1^{3^{e-1}},\alpha,s_2\in H$. Furthermore, since $[s_1,\beta]=s_2$ and $s_2$ commutes with $s_1$, an elementary fact on commutators, see for instance \cite[Hilfssatz 1.3 (a)]{huppertI1967}, yields $[s_1^{3^{e-1}},\beta]=[s_1,\beta]^{3^{e-1}}=s_2^{3^{e-1}}=1$  whence $[s_1^{3^{e-1}},\beta]=1$ follows. Similarly, for Case B(3), $|G|=2^{2e}$, and $Z(G)=\langle s_2^{\nu3^{e-1}} \rangle$.
Therefore, $\bar G$ is isomorphic to B(3) for $|\bar{G}|=2^{2e}$, and to A(3) for $|\bar{G}|=2^{2(e-1)+1}$. But this contradicts  \eqref{casesbarg}. Thus $H$ is not abelian.
\subsubsection*{Case of non-abelian $H$} If $G$ has maximal nilpotency class, then the claim follows from Result \ref{qu14sep}. Otherwise, from the initial part in the proof of Result \ref{qu24agosto}, one of the cases in (I) or (II) occurs. Since $\bar G=G/Z(G)$ is given by \eqref{casesbarg}, $G$ must be isomorphic to either I(1) or II(1) in the proof of Result \ref{resbl1}; see the proof of \cite[Theorem 3.8]{Xu2008}. In both cases, the proof of Result \ref{qu24agosto} shows that $Z(G)$ has more than $3$ elements, for instance, $1,x,x^2$ and $s_2^{3^{e-1}}$  for I(1) and $s_1^{3^{e-1}}$ for II(1). But this contradicts our assumption $|Z(G)|=3$.
\end{proof}
\end{theorem}
\section{An infinite family of elliptic type extremal $3$-Zomorrodian curves}
\label{ss29luglioA}
We describe an explicit construction that provides an elliptic type extremal $3$-Zomorrodian curve for every possible genus.
We keep notation from Subsection \ref{pre} and suppose that $|\bar{G}|=3^{h+1}$ with $h\geq 2$.

Since $\bar{G}$ is generated by two elements, the factor group $\bar{G}/\Phi(\bar{G})$ is an elementary abelian group of order $9$ where $\Phi(\bar{G})$ is the Frattini subgroup of $\bar{G}$.
As $\bar{H}$ is a maximal subgroup, $\Phi(\bar{G})$ is contained in $\bar{H}$.

Let $\theta_1$ be the $\Phi(\bar{G})$-orbit containing $\bar{P}=(-1,0,1)$. Then $|\theta_1|=3^{h-1}$. Since $\Phi(\bar{G})$ is a normal subgroup of $\bar{H}$, the $H$-orbit $\theta$ containing $\bar P$ is partitioned into three $\Phi(\bar{G})$-orbits which may be parameterized by $\Phi(\bar{G})$ together with its two cosets in $\bar{H}$.

More precisely, $\theta_1=\{\bar{f}(\bar P)|\bar{f}\in \Phi(\bar{G})\}$ while the other two $\Phi(\bar{G})$ orbits are $\theta_2=\{\bar{u}\bar{f}(\bar{P})|\bar{f}\in \Phi(\bar{G})\}$ and $\theta_3=\{\bar{u}^2\bar{f}(\bar{P})|\bar{f}\in \Phi(\bar{G})\}$ where $\bar{u}\in \bar{H}$ with $\bar{u}^3\in \Phi(\bar{G})$. For a proof of the following lemma, the Reader is referred to \cite{yu}, or \cite[Section 4]{knp}.
\begin{lemma}
\label{lem26jun} If $\bar{Q}\in \theta_2$ then the line through $\bar{P}$ and $\bar{Q}$ meets $\bar{\cE}$ in a point $\bar{R}\in \theta_3$.
\end{lemma}
Now, take a line $\ell$ through $\bar{P}$ and a point $\bar{Q}\in \theta_2$. Then $\ell$ has homogenous equation
$mX-Y+mZ=0$ for some $m\in \mathbb{K}$. Furthermore, the (inflectional) tangent to $\bar{\cE}$ at $\bar{P}$ has homogenous equation $X+Z=0$. In the function field $\mathbb{K}(\bar{\cE})=\mathbb{K}(\bar{x},\bar{y})$ with $\bar{x}^3+\bar{y}^3+1=0$, the rational function $$\bar{t}=\frac{m\bar{x}-\bar{y}+m}{\bar{x}+1}$$
has one pole (with multiplicity $2$), namely $\bar{P}$, and two zeros (both of multiplicity $1$), one is $\bar{Q}$ while the other $\bar{R}$ is in $\theta_3$ by Lemma \ref{lem26jun}. Let
$$\bar{w}=\prod_{\bar{f}\in \Phi(\bar{G})}\bar{f}(\bar{t}).$$
Then the poles of $\bar{w}$, each with multiplicity $2$, are exactly the points in $\theta_1$ while the zeros of $\bar{w}$, each with multiplicity $1$ are exactly the points in $\theta_2\cup \theta_3$. From this, the following result follows.
\begin{lemma}
\label{lem6agosto2018} There is no element in $\mathbb{K}(\bar{\cE})$ whose cube is equal to $\bar{w}$.
\end{lemma}
Therefore, we define the curve $\cX$ as those whose function field $\mathbb{K}(\cX)$ is the Kummer extension of $\mathbb{K}(\bar{\cE})$ defined by $z^3=\bar{w}$. Our goal is to show that $\cX$ is an extremal $3$-Zomorrodian curve.

Since no nontrivial element in $\bar{H}$ fixes a point in $\theta$, the above discussion also gives the following result.
\begin{lemma}
\label{lem26Ajun} For any $\bar{g}\in \bar{G}$, the rational function $\bar{g}(\bar{w})/\bar{w}$ is either constant, or its poles, each with multiplicity $3$, are exactly the points of one of the $\Phi(\bar{G})$-orbits $\theta_1,\theta_2,\theta_3$, and its zeros, each with multiplicity $3$, are exactly the points of another of these  $\Phi(\bar{G})$-orbits.
\end{lemma}
To find a large $3$-subgroup in $\aut(\cX)$ the following result is useful.
\begin{lemma}
\label{lemAA}
For any $\bar{g}\in \bar{G}$, there exists a rational function $\bar{v}\in \mathbb{K}(\bar{\cE})$ such that $\bar{v}^3=\bar{g}(\bar{w})/\bar{w}$.
\end{lemma}
\begin{proof} For $i=1,2,3$, the sum of all points in $\theta_i$ can be viewed as a divisor $D_i$ of $\mathbb{K}(\bar{\cE})$. Then $\deg(D_i)=|\theta_i|=3^{h-1}$, and the complete linear series $|D_i|$ has (projective) dimension $3^{h-1}-1$ by the Riemann-Roch Theorem. Therefore, there exists $\bar{Q}\in D_2$ such that
\begin{equation}
\label{eq26jun} D_1\equiv D_2-\bar{Q}+\bar{S}
\end{equation}
 where $\bar{S}$ is a point in $\bar{\cE}$.
Since $\theta$ is a $\bar{G}$-orbit, some element $\bar{g}\in \bar{G}$ takes $\bar{P}$ to $\bar{R}$. Then $\bar{G}_{\bar{Q}}$ is conjugate to $\bar{G}_{\bar P}$ and hence it contains a nontrivial element
$\bar{g}^*$. Since $\bar{g}^*(\bar{Q})=\bar{Q}$ and the fixed points of $\bar{g}^*$ form a $Z(\bar{G})$-orbit, each of the three fixed points of $\bar{g}^*$ is in $\theta_2$. Furthermore,
$\bar{g}^*$ preserves each of the $\Phi(\bar{G})$-orbits $\theta_1,\theta_2,\theta_3$.
Now, from (\ref{eq26jun}), $\bar{g}^*(D_1)\equiv \bar{g}^*(D_2)-\bar{g}^*(\bar{Q})+\bar{g}^*(\bar{S})$ whence
\begin{equation}
\label{eq26Ajun} D_1\equiv D_2-\bar{Q}+\bar{g}^*(\bar{S}).
\end{equation}
This together with (\ref{eq26jun}) yield $\bar{S}\equiv \bar{g}^*(\bar{S})$. Since $\bar{\cE}$ is not rational, $\bar{S}=\bar{g}^*(\bar{S})$ follows. Therefore, $D_1\equiv D_2$, and hence
there exists a rational function $\bar{v}\in \mathbb{K}(\bar{\cE})$ such that $\Div(\bar{v})=D_1-D_2$. From Lemma \ref{lem26jun}. the claim follows.
\end{proof}
Lemma \ref{lemAA} shows that for every $\bar{g}\in \bar{G}$, the map $g$ defined by
\begin{equation}
g(\bar{x},\bar{y},z)=(\bar{g}(\bar{x}),\bar{g}(\bar{y}),\bar{v}z)
\end{equation}
is in $\aut(\cX)$. They are as many as $|\bar{G}|$ and form a subgroup of $\aut(\cX)$. The map
\begin{equation}
c(\bar{x},\bar{y},z)=(\bar{x},\bar{y},\varepsilon z)
\end{equation}
is also in $\aut(\cX)$. Therefore $G$ together with its two cosets $cG$ and $c^2G$ give a subgroup $G$ of $\aut(\cX)$  of order $3|\bar{G}|=3^{h+2}$. On the other hand, since the poles of $\bar{w}$ have multiplicity $2$ and are the points in $\theta_1$ while the zeros of $\bar{w}$ have multiplicity $1$ and they are precisely the points in $\theta_2\cup\theta_3$, a basic fact on Kummer extensions, see for instance \cite[Corollary III. 7.4]{stichtenoth1993}, states that $\gg(\cX)-1=\ha (2|\theta_1|+2|\theta_2|+2|\theta_3|)=3^h$.
Therefore, $\cX$ is an elliptic type extremal $3$-Zomorrodian curve of genus $\gg(\cX)=3^h+1$ such that $G$ is a Sylow $3$-subgroup of $\aut(\cX)$ and the following result is proven.
\begin{lemma}
\label{lem8agosto2018}
The group $G$ has the following properties.
\begin{itemize}
\item[(i)] There exists an order $3$ subgroup $Z$ of $Z(G)$ such that the quotient group $G/Z$ is of maximal nilpotency class and contains an abelian maximal subgroup.
\item[(ii)] Either $\Phi(G)\cong C_{3^n}\times C_{3^n}\times Z$ with $h=2n+1$, or $\Phi(G)\cong C_{3^n}\times C_{3^{n-1}}\times Z$ with $h=2n$.
\item[(iii)] If $h>2$ then $Z(G)\cong C_3\times C_3$.
\item[(iv)] For $h>3$, in terms of generators and relations $G$ is the group given in Result \ref{qu24agosto}.
\end{itemize}
\end{lemma}
\begin{proof} Claim (i): Let $Z$ be the subgroup of $Z(G)$ generated by $c$. Then $\bar{G}=G/Z$ is a subgroup of $\aut(\bar{\cX})$ and some non-trivial element of $\bar{G}$ fixes a point in $\bar{\cX}$. Therefore, the claim follows from Lemma \ref{lem28jun}.

Claim (ii): From the definition of $\bar{w}$, we have $\bar{f}(\bar{w})=\bar{w}$ when $\bar{f}\in \Phi(\bar{G})$. Therefore, for every $\bar{f}\in \Phi(\bar{G})$,
\begin{equation}
f(\bar{x},\bar{y},z)=(\bar{f}(\bar{x}),\bar{f}(\bar{y}),z)
\end{equation}
is in $\aut(\cX)$, and hence such automorphisms form a group $F\cong \Phi(G)$. Therefore, $\Phi(G)=F\times Z$, and (ii) follows from Proposition \ref{prop18jul2018PG}.

Claim (iii): $Z(G)$ contains both $c$ and $\beta$ inducing $\bar{\beta}$ on $\aut(\bar \cE)$, where $\bar{\beta}$ is as defined in Section \ref{pre}, and hence $|Z(G)|\geq 9$. On the other hand, $|\bar{G}|=3$. Hence (iii) holds. If $h>3,$ Claim (iv) comes from Theorem \ref{the25agosto}.
\end{proof}
As a corollary of the results proven in the present subsection, the following result is obtained.
\begin{proposition}
\label{pro6agosto2018} For every $h\geq 3$ there exists an elliptic type extremal $3$-Zomorrodian curve $\cX$ of genus $3^h+1$ such that a Sylow $3$-subgroup of $\aut(\cX)$ has the properties
listed in Lemma \ref{lem8agosto2018}.
\end{proposition}
\begin{rem}
\label{rem11agosto2018}
{\em{We point out that the curve $\cX$ in Proposition \ref{pro6agosto2018} gives rise up to at least two elliptic type extremal $3$-Zomorrodian curves of genus $3^{h-1}+1$. According to (iii) of Lemma \ref{lem8agosto2018}, $Z(G)$ has four subgroups of order $3$. From Lemma \ref{nuovo}, just one of them, say $U$, is such that $|Z(G/U)|=9$. The arising quotient curve $\cX/U$ is not elliptic, and hence an elliptic type extremal $3$-Zomorrodian curve such that a Sylow $3$-subgroup of $\aut(\cX/U)$ has center of order $9$. Now, choose for $U$  one of other three subgroups of $Z(G)$ of order $3$. From Lemma \ref{nuovo}, $|Z(G/U)|=3$. If $U=Z$ where $Z$  as in (ii) of Lemma \ref{lem8agosto2018}, then the arising quotient curve the above elliptic curve $\bar{\cX}$. For $U\neq Z$, Proposition \ref{lem13jan2018} states that the quotient curve $\cX/U$ has genus $\geq 2$. From Proposition \ref{prop29luglio2018} it is an elliptic type extremal $3$-Zomorrodian curve of genus $3^{h-1}+1$ such that a Sylow $3$-subgroup of $\aut(\cX/U)$ has center of order $3$.  }}
\end{rem}
We illustrate our construction for the smallest case.
\begin{example}
\label{excaseg=10} {\em{ Let $\bar{G}$ be the linear group $U(3,3)$ of order $27$ generated by $\bar{\alpha}$ and $\bar{\delta}:\,(X,Y,Z)=(Y,Z,X)$. Also, let $J(\bar{\cE}) \cap \bar{G}=\langle \bar{\alpha},\bar{\delta}\rangle$, and $Z(\bar{G})=\Phi(\bar{G})=\bar{G}'=\langle \bar{\beta} \rangle.$
Let $$\bar{P}_3=(1,-1,0),\bar{P}_4=(1,-\varepsilon,0),\bar{P}_5=(1,-\varepsilon^2,0),\bar{P}_6=(0,1,-1),\bar{P}_7=(0,1,-\varepsilon),\bar{P}_8=(0,1,-\varepsilon^2).$$
Then $\theta_1=\{\bar P, \bar{P}_1,\bar{P}_2\},\,\theta_2=\{\bar{P}_3,\bar{P}_4,\bar{P}_5\}$ and $\theta_3=\{\bar{P}_6,\bar{P}_7,\bar{P}_8\}$. Now, take $\bar{P}_3$ for $\bar Q$. Then $m=\varepsilon$, as $\ell$ has equation $\varepsilon X+Y+\varepsilon Z=0$, and  $\bar{t}=(\varepsilon \bar{x}+\bar{y}+\varepsilon)/(\bar{x}+1)$. By a straightforward computation, $\bar{w}=\bar{x}/\bar{y}^2$. Hence  $\K(\cX)$ is given by
\begin{equation}
\label{eq12agosto}
\left \{
\begin{array}{lll}
\bar{x}^3+\bar{y}^3+1=0;\\
z^3=\frac{\bar{x}}{\bar{y}^2}.
\end{array}
\right.
\end{equation}
Furthermore, $\bar{\alpha}(w)/w=-\varepsilon$ and $\bar{\delta}(w)/w=\bar{y}^3$. Hence $G=\langle \alpha, \delta \rangle$ where $\alpha(\bar{x},\bar{y},z)=(\bar{x},\varepsilon\bar{y},-\varepsilon^2 z)$, and $\delta(\bar{x},\bar{y},z)=(\bar{y}/\bar{x},1/\bar{x},\bar{y})$. Therefore, $G\cong$``SmallGroup''(81,9). Elimination $x$ from (\ref{eq12agosto}) gives $z^9\bar{y}^3+\bar{y}^3+1=0$ which is an (affine) equation of $\cX$ regarded as a plane curve. A straightforward computation shows that $\cX$ is isomorphic to Example \ref{ex27Aluglio}.
}}
\end{example}
Computations carried out by MAGMA package where $\K$ is viewed as the algebraic closure of the finite field $\mathbb{F}_q$ can provide explicit equations for $\cX$ and show the existence of elliptic type extremal $3$-Zomorrodian curves with Sylow $3$-subgroup of maximal nilpotency class.
Here we limit ourselves to $\gg(\cX)=82,244$.

Let $h=5$. The smallest prime $q$ such that $\aut(\bar{\cE})$ has a subgroup of order $|\bar{G}|=3^6$ defined over $\mathbb{F}_q$ is $q=271$. This allows us to find an explicit equation for $\cX$ over $\mathbb{F}_{271}$, in terms of $\K(\cX)=\K(x,y,z)$ we have
$$\left \{
\begin{array}{lll}
x^3+y^3+1=0;\\
z^3=(y^{54} + 9y^{51} + 151y^{48} + 191y^{45} + 243y^{42} + 21y^{39} + 86y^{36} + 184y^{33} + y^{30} + 153y^{27} + y^{24} +
184y^{21} + \\86y^{18} + 21y^{15} +243y^{12} + 191y^9 + 151y^6 + 9y^3 + 1)/(y^{53} + 9y^{50} + 261y^{47} +258y^{44} + 138y^{41} +
146y^{38} + \\ 206y^{35} + 24y^{32} + 12y^{29} + 12y^{26}+ 24y^{23} + 206y^{20} + 146y^{17} + 138y^{14} + 258y^{11} + 261y^8 + 9y^5+ y^2)x.
\end{array}
\right.
$$
Furthermore, $G$ has order $2187$ and its center $Z(G)$ is isomorphic to $C_3\times C_3$. Since $\beta$ inducing $\bar{\beta}$ is in $Z(G)$, the arising quotient curve $\tilde{\cX}=\cX/\langle {\beta}\rangle$ is an elliptic type extremal $3$-Zomorrodian curve of genus $\gg(\tilde{\cX})=82$ with function field $\K(\xi,\eta,\zeta)$ where
$$\left \{
\begin{array}{lll}
\xi^3+\eta^2+\eta=0;\\
\zeta^3=(\eta^{18} + 9\eta^{17} + 151\eta^{16} + 191\eta^{15} + 243\eta^{14} + 21\eta^{13} + 86\eta^{12} + 184\eta^{11} + \eta^{10} + 153\eta^9 + \eta^8 + 184\eta^7 + \\
86\eta^6 + 21\eta^5 +243\eta^4 + 191\eta^3 +151\eta^2 + 9\eta^1 + 1)/(\eta^{17} + 9\eta^{16} + 261\eta^{15} + 258\eta^{14} + 138\eta^{13} +146\eta^{12} + \\
206\eta^{11} + 24\eta^{10} + 12\eta^9 + 12\eta^8 + 24\eta^7 +206\eta^6 +146\eta^5 + 138\eta^4 + 258\eta^3 + 261\eta^2 + 9\eta +1)(\xi/\eta).
\end{array}
\right.
$$
A Sylow $3$-subgroup of $\aut(\tilde{\cX})$ is isomorphic to ``SmallGroup''(729,100). Another maximal quotient of $G$  is isomorphic to ``SmallGroup''(729,95) which contains an abelian subgroup of index $3$. Hence the arising quotient curve is elliptic.
The forth maximal quotient of $G$ is isomorphic to ``SmallGroup''(729,40) whose center is an elementary abelian group of order $9$.

Let $h=4$. This time we choose $q=73$. Then $\aut(\bar{\cE})$ has a subgroup of order $729$ defined over $\mathbb{F}_{729}$ and our construction provides an elliptic type
extremal $3$-Zomorrodian curve $\cX$ of genus $\gg(\cX)=82$. In terms of $\K(\cX)=\K(x,y,z)$,
$$\left \{
\begin{array}{lll}
x^3+y^3+1=0;\\
z^3=(y^{18} + 3y^{15} + 52y^{12} + 26y^9 + 52y^6 + 3y^3 + 1)/(y^{17} + 3y^{14} +
        5y^{11} + 5y^8 + 3y^5 + y^2)x.
\end{array}
\right.
$$
A Sylow $3$-subgroup $G$ is isomorphic to the ``SmallGroup''(729,40). Its four subgroups of index $3$ are isomorphic to ``SmallGroup''(243,53),  ``SmallGroup''(243,15), ``SmallGroup''(243,2), ``SmallGroup''(243,53), respectively.  None of these are abelian whereas ``SmallGroup''(243,2) is the unique minimal non-abelian.

The four quotient groups of $G$ arising from the order $3$ subgroups of $Z(G)$ are
$\bar{G}_1\cong$``SmallGroup''(243,3), $\bar{G}_2\cong$``SmallGroup''(243,26), $\bar{G}_3\cong$``SmallGroup''(243,28), and $\bar{G}_4\cong$``SmallGroup''(243,28) respectively. Here, $|Z(\bar{G}_1)|=9,
|Z(\bar{G}_2)|=|Z(\bar{G}_3)|=|Z(\bar{G}_4)|=3$, and just one of them, namely $\bar{G}_2$, has an abelian subgroup of index $3$.

\section{Elliptic type extremal $3$-Zomorrodian curves of low genus}
\label{lowex}
For $\gg(\cX)=10,28$, Lemma \ref{keylemma} together with Result \ref{rescor3.5} are enough to determine the structure of $G$. Here, we give some more detail.
\subsection{Genus=28}
 \label{sub28} Let $q=73$ and $\cX$ as in case $h=4$;  see Section \ref{ss29luglioA}. Four quotient curves arise from the order $3$ subgroups of $Z(G)$, say $\bar{\cX}_i$, with $\bar{G}_i\le \aut(\bar{\cX}_i)$. Here,  $\bar{\cX}_2$ is elliptic, the other three being elliptic type extremal $3$-Zomorrodian curves of genus $28$. At least two of them, $\bar{\cX}_1$ and $\bar{\cX}_3$, are non-isomorphic. The subgroup $\bar{G}_1\cong$``SmallGroup''(243,3) of $\aut(\bar{\cX}_1)$ has no abelian but two minimal non-abelian subgroups.  The subgroup $\bar{G}_3\cong$``SmallGroup''(243,28) of $\aut(\bar{\cX}_3)$ has no abelian but one minimal non-abelian subgroup. The latter one holds for $\bar{\cX}_4$ which may be isomorphic to $\bar{\cX}_3$.

Apply the construction in Section \ref{ss29luglioA} for $h=3$, and choose $q=19$. Then the resulting curve $\cX$ is, in terms of $\K(\cX)=\K(x,y,z)$,
$$\left \{
\begin{array}{lll}
x^3+y^3+1=0;\\
z^3=(y^6+y^3+1)/(y^5+y^2)x.
\end{array}
\right.
$$
The following maps are in $\aut(\cX)$
$${\mbox{$f(x,y,z)=\big(\frac{4y}{y^3 + 12}x^2 + \frac{6}{y^3 + 12}x + \frac{4y^2}{y^3 + 12},\,\frac{3}{y^3 + 12}x^2 + \frac{2y^2}{y^3 + 12}x + \frac{3y}{y^3 + 12},\,\frac{13y}{y^3+12}xz\big)$;}}$$
$$g(x,y,z)=(7x,7y,16z).$$
and they generate a group $G$ isomorphic to ``SmallGroup''(243,3). Its four subgroups of index $3$ are isomorphic to ``SmallGroup''(81,3),  ``SmallGroup''(81,3), ``SmallGroup''(81,12), ``SmallGroup''(81,12), respectively. The four quotient groups of $G$ arising from the order $3$ subgroups of $Z(G)$ are isomorphic to the non-abelian groups
$\bar{G}_1\cong$``SmallGroup''(81,7), $\bar{G}_2\cong$``SmallGroup''(81,7), $\bar{G}_3\cong$``SmallGroup''(81,9), $\bar{G}_4\cong$``SmallGroup''(81,9), respectively. For $1\le i \le 4$, $|Z(\bar{G}_i)|=3$, and  $\bar{G}_i$ is of maximal nilpotency class.
\subsection{Genus=10} \label{sub10}
Let $q=19$. We have four quotient curves $\bar{\cX}_i$ with $\bar{G}_i\le \aut(\bar{\cX}_i)$ where $\bar{G}_i$ is defined at the end of Subsection \ref{sub28}. Since ``SmallGroup''(81,7) has an elementary abelian group of order $27$ but ``SmallGroup''(81,9) does not, we have that $\bar{\cX}_i$ is elliptic only for either $i=3$, or $i=4$, say $i=4$. Therefore $\bar{\cX}_i$ is an elliptic type extremal $3$-Zomorrodian curve of genus $10$ with a Sylow $3$-subgroup isomorphic to ``SmallGroup''(81,7) for $i=1,2$ and to ``SmallGroup''(81,9) for $i=3$. Choose that $i$ with $i\in\{1,2,3\}$ for which $\bar{G}_i=G/Z$ with $Z=\langle \beta \rangle$ and $\beta(x,y,z)=(\epsilon x,\epsilon^2 y,z)$. Then $\K(\bar{\cX}_i)=
\K(\xi,\eta,\zeta)$ where $\xi=xy,\eta=y^3,\zeta=z$, that is,
$$\left \{
\begin{array}{lll}
\xi^3+\eta^2+\eta=0;\\
\zeta^3=\frac{\eta^2+\eta+1}{\eta^2+\eta}\xi.
\end{array}
\right.
$$
Then $\aut(\cX_i)$ has a Sylow $3$-subgroup isomorphic to  ``SmallGroup''(81,9), and hence $i=3$. Actually, $\cX$ is isomorphic to Example \ref{eq12agosto} (equivalently to Example \ref{ex27Aluglio}).
Similarly, a MAGMA aided computation shows that $\K(\cX_1)$ is
$$\left \{
\begin{array}{lll}
\xi^3+\eta^2+\eta=0;\\
\zeta^3=-\frac{\xi^3+\xi^2+\xi}{\eta^2};
\end{array}
\right.
$$
Elimination of $\xi$ from (\ref{eq12agosto}) gives $\eta^7 -2\eta^4 + \eta^3\zeta^3 + \eta^2\zeta^6 + \eta -\zeta^3=0$ which is an (affine) equation of $\cX_1$ regarded as a plane curve.

\section{An infinite family of non-elliptic type extremal $3$-Zomorrodian curves}
\label{ss29luglioB}
The elliptic type extremal $3$-Zomorrodian curve given in Example \ref{ex27Aluglio} is projectively equivalent to the irreducible plane curve $\cX_0$ with affine equation $y^9+x^6+x^3=0$.
The singular points of $\cX_0$ are the origin $O=(0,0)$ and its unique point at infinity $X_\infty=(1,0,0)$, both are triple points. For any pair $(\lambda,\mu)$ with $\lambda^3=\mu^9=1,$ and $\lambda,\mu\in \mathbb{K}$ the map
$\alpha_{\lambda,\mu}(x,y)=(\lambda x, \mu y)$ is in $\aut(\cX_0)$.
All these maps form an abelian subgroup $A$ of $\aut(\cX_0)$ of order $27$. Moreover, $\aut(\cX_0)$ contains the map $$\alpha_2(x,y)=\bigg(\frac{x}{y^3}, \frac{x}{y^2}\bigg),$$ which has order $3$.
The subgroup $G$ of $\aut(\cX_0)$ generated by $A$ together with $\alpha_2$ is the semidirect product $G=A\rtimes \langle \alpha_2\rangle$, and it has order $81$. The center of $G$ has order $3$ and
it is generated by $\alpha_{1,\varepsilon}$ where $\varepsilon$ denotes a primitive third of unity. From the proof of Proposition \ref{prop29luglio2018}, the set $\Omega$ of fixed points of $\alpha_{1,\varepsilon}$ has size $\frac{1}{9}|G|=9$. Since $\alpha_{1,\varepsilon}$ fixes the (non-singular) points $R_i=(-\varepsilon^i,0)$ for $i=1,2,3$, the remaining $6$ points of $\Omega$ are branches (places) with center at the origin $O=(0,0)$, or at the infinite point $X_\infty$ of the $x$-axis. Since both are triple points, the only possibility is that each such point is the center of exactly three branches. Let $P_1,P_2,P_3$ denote the branches centered at $O$, and $Q_1,Q_2,Q_3$ those centered at $X_\infty$. Then the $y$-axis (resp. the line at infinity) is the common tangent to the branches $P_i$ (resp. $Q_i$).

Write, for brevity, $F$ in place of $\mathbb{K}(\cX_0),$ and regard $K=\mathbb{K}(x^3)$ as a subfield of $F$. A straightforward computation shows that $x^3$ is fixed by $A$. Since $[F:K]=27$, the fixed field ${\rm{Fix}}(A)$ of $A$ coincides with $K$. Furthermore, \cite[Theorem 6.42]{HKT} applied to $x$ and $y$, viewed as functions in $F$, gives
$$\div (x)=-3(Q_1+Q_2+Q_3)+3(P_1+P_2+P_3), \quad \div(y)=-2(Q_1+Q_2+Q_3)+P_1+P_2+P_3+R_1+R_2+R_3.$$

Now, take from $F$ the function
\begin{equation}
\label{eq30sep}
t=\frac{x^9-3x^3-1}{x^3(x^3+1)}=x^3+\frac{x^3}{y^9}+\frac{y^9}{x^6}=x^3+\alpha_2(x^3)+\alpha_2^2(x^3).
\end{equation}
A straightforward computation shows that $t$ is fixed by both $A$ and $\alpha_2$. Therefore, $K(t)$ is a subfield of the fixed field $\rm{Fix}(G)$ of $G$. Furthermore,
$$\div(t)_{\infty}=9(P_1+P_2+P_3+Q_1+Q_2+Q_3+R_1+R_2+R_3).$$
Since $[F:\K(t)]=81$, this shows that $\K(t)$ is ${\rm{Fix}}(G)$.
\begin{lemma} \label{lemtow1}
Let $L$ be a function field over $\K$ containing $F$. If $\gamma \in \aut(L)$ fixes $t$, then $\gamma$ preserves $F$.
\end{lemma}
\begin{proof}
We first show $\gamma(K)=K$. Let $\xi=x^3$ and $\gamma(\xi)=\theta$. Then (\ref{eq30sep}) gives $\gamma(t)=(\theta^3-3\theta-1)/(\theta^2+\theta),$ and from
$\gamma(t)=t$ we infer $\gamma(t)=(\xi^3-3\xi-1)/(\xi^2+\xi).$ Therefore,
$$(\theta^3-3\theta-1)(\xi^2+\xi)-(\theta^2+\theta)(\xi^3-3\xi-1)=(\theta-\xi)(\theta\xi+\xi+1)(\theta\xi+\theta+1)=0.$$
From this, either $\theta=\xi$, or $\theta=-(\xi+1)/\xi$ or $\theta=-1/(\xi+1)$. These three possibilities are treated separately.

If $\gamma(\xi)=\theta=\xi$ then $\gamma(x)^3=x^3$ whence $\gamma(x)=\varepsilon^i x$ with $0\le i \le 2$. Hence $\gamma(x)\in F$. Moreover,
$$0=\gamma(y^9)+\gamma(x^6)+\gamma(x^3)=\gamma(y)^9+\gamma(\xi)^2+\gamma(\xi)=\gamma(y)^9+\xi^2+\xi=\gamma(y)^9+y^9,$$
whence $\gamma(y)=\nu y$ with $\nu^9=1$. Thus $\gamma(y)\in F$, and hence $\gamma(F)=F$.

If $\gamma(\xi)=\theta=-(\xi+1)/\xi$ then $\gamma(x)^3=-(x^3+1)/x^3$. Then $\gamma(x)^3=(y^3/x^2)^3$, whence $\gamma(x)=\varepsilon^i (y^3/x^2)$ with $0\le i \le 2$. Thus $\gamma(x)\in F$.
Moreover, from $y^9+x^6+x^3=0$,
$$0=\gamma(y^9)+\gamma(x^6)+\gamma(x^3)=\gamma(y)^9+\frac{1+x^3}{x^3}=\gamma(y)^9+\bigg(\frac{y}{x}\bigg)^9=0,$$ whence $\gamma(y)=\nu (-x/y)$ with $\nu^9=1$. Thus $\gamma(y)\in F$, and hence $\gamma(F)=F$.

If $\gamma(\xi)=\theta=-1/(\xi+1)$ then $\gamma(x)^3=-1/(1+x^3)$. Hence $\gamma(x)^3=(x/y^3)^3$. Thus $\gamma(x)=\varepsilon^i (x/y^3)$ with $0\le i \le 2$. Hence $\gamma(x)\in F$.
Moreover, from $y^9+x^6+x^3=0$,
$$0=\gamma(y^9)+\gamma(x^6)+\gamma(x^3)=\gamma(y)^9-\frac{x^3}{(1+x^3)^2}=\gamma(y)^9-\bigg(\frac{x}{y^2}\bigg)^9=0,$$ whence $\gamma(y)=\nu (x/y^2)$ with $\nu^9=1$. So $\gamma(y)\in F$, and hence $\gamma(F)=F$.
\end{proof}

We show how to use Lemma \ref{lemtow1} for the construction of a tower of curves $\{\cX_\ell\}$ with $\ell=1,2,\ldots,$ consisting of extremal $3$-Zomorrodian curves of non elliptic type. We change our notation by setting $F_0=\cX_0$ where $\cX_0$ is the above elliptic type extremal $3$-Zomorrodian curve of genus $\gg(\cX_0)=10$ given by the affine equation $y^9+x^6+x^3=0$.  Let $t\in F_0$ be defined as in (\ref{eq30sep}).

Now, $F_0$ has a unique maximal unramified abelian $3$-extension $L$ in a fixed algebraic closure $\bar{F}_0$ of $F_0$. This means that
\begin{itemize}
\item[(i)] $L|F_0$ is a Galois extension of degree $3^{2\gg(F_0)}$,
\item[(ii)] $L|F_0$ is an unramified extension,
\item[(iii)] ${\rm{Gal}}(L|F_0)$ is an elementary  abelian $3$-group,
\end{itemize}
see, for instance, the survey \cite[Section 4.7]{ps}. It should be noticed that if both the extensions $L|F$ and $F|\mathbb{K}(t)$ are Galois extensions, the extension $L|\mathbb{K}(t)$ needs not be Galois.
Let $L'$ the Galois closure of $F$ over $\K(t)$ in $\bar{F}_0$, and set $\Gamma={\rm{Gal}}(L^\prime|\K(t))$. Then the diagram shows the fields extensions with the relative dimensions which we deal with:
\[
\newcommand{\ext}[1]{
  \hphantom{\scriptstyle#1}\bigg|{\scriptstyle#1}%
}
\begin{array}{@{}c@{}}
L^\prime \\
\ext{|\Gamma|/(81 \cdot 3^{2\gg(F_0)})} \\
L \\
\ext{3^{2\gg(F_0)}} \\
F_0 \\
\ext{81} \\
\K(t)
\end{array}
\]
Therefore,
$|\Gamma|=[L^\prime:\mathbb{K}(t)]=[L^\prime:L][L:F_0][F_0:\mathbb{K}(t)]=[L^\prime:L]\cdot 3^{2\gg(F_0)}\cdot 81=[L^\prime:L] \cdot 3^{2\gg(F_0)+4}.$
This implies that $3^{2\gg(F_0)+4}$ divides $|\Gamma|$. Since $\Gamma$ fixes $t$,  Lemma \ref{lemtow1} yields that $\Gamma$ preserves $F_0$. Since $L$ is the unique maximal unramified abelian $3$-extension of $F_0$ in $\bar{F}_0$, this shows that $\Gamma$ preserves $L$, as well. Let $G$ be the subgroup of $\aut(L)$ induced by $\Gamma$ on $L$. Then $G\cong\Gamma/\Sigma $ where $\Sigma$ is the subgroup of $\Gamma$ fixing $L$ elementwise. Since $|\Sigma|$ divides $[L^\prime:L]$, it turns out that $G$ contains a $3$-subgroup of order $3^{2\gg(F_0)+4}$. On the other hand, the Hurwitz genus formula applied to the unramified extension $L|F$ gives
$$\gg(L)-1=3^{2\gg(F_0)}(\gg(F_0)-1)=3^{2\gg(F_0)+2}.$$
From Theorem \ref{proCdic112017}, $|G|=3^{2\gg(F_0)+4}$.
Therefore, the following result holds.
\begin{lemma}
\label{lem30sep} The curve $\cY$ with function field $L$ is an extremal $3$-Zomorrodian curve such that $G$ is a Sylow $3$-subgroup of $\aut(\cY)$.
\end{lemma}
As a byproduct, $[L:\mathbb{K}(t)]=[L:F_0][F_0:\mathbb{K}(t)]=3^{2\gg(F_0)}3^4=|G|$ which  shows that $L|\mathbb{K}(t)$ is a Galois extension.
\begin{lemma}
\label{lemA30sep} The extremal $3$-Zomorrodian curve $\cY$ in Lemma \ref{lem30sep} is not of elliptic type.
\end{lemma}
\begin{proof} Let $G_0$ be the subgroup of $G$ which fixes $\cX_0$ pointwise. Then the quotient group $G/G_0$ is an elementary abelian $3$-group of order $3^{2\gg(F_0)}\ge 9$. Therefore, $G/G_0$ has more than four maximal subgroups of index $3$, and this holds true for $G$. From Lemma \ref{lemC29luglio2018}, $\cX_0$ is not of elliptic type.
\end{proof}
The above construction also works if $L$ is taken for the unique maximal unramified abelian $3^m$-extension of $F_0$ in $\bar{F}_0$ where $m$ is any positive integer.
Here $K$ has the following properties.
\begin{itemize}
\item[(i)] $L|F_0$ is a Galois extension of degree $3^{2\gg(F_0)m}$,
\item[(ii)] $L|F_0$ is an unramified extension,
\item[(iii)] ${\rm{Gal}}(L|F_0)$ is the direct product of $2\gg(F_0)$ cyclic groups of order $3^m$.
\end{itemize}
Therefore, the construction provides an infinite family of non elliptic type extremal $3$-Zomorrodian curves.

\vspace{0.5cm}\noindent {\em Authors' addresses}:

\vspace{0.2cm}\noindent G\'abor KORCHM\'AROS and Maria MONTANUCCI\\ Dipartimento di
Matematica, Informatica ed Economia\\ Universit\`a degli Studi  della Basilicata\\ Contrada Macchia
Romana\\ 85100 Potenza (Italy).\\E--mail: {\tt
gabor.korchmaros@unibas.it} and {\tt maria.montanucci@unibas.it}.

\begin{thebibliography}{30}
\bibitem{BB1} Y.~Berkovich and J. Zvonimir, {\emph{Groups of Prime Power Order,}} Vol. I, Walter de Gruyter GmbH \& Co. KG, Berlin, (2008), xx+512 pp.
\bibitem{BB2} Y.~Berkovich, Z.~Janko, {\emph{Groups of Prime Power Order,}} Vol II, Walter de Gruyter GmbH \& Co. KG, Berlin, (2008), xvi+596 pp.
\bibitem{BB3} Y.~Berkovich, Z.~Janko, {\emph{Groups of Prime Power Order,}} Vol III, Walter de Gruyter GmbH \& Co. KG, Berlin, (2011), xxvi+639 pp.
\bibitem{Bl1} N. Blackburn, On a special class of p-groups, \emph{Acta Math.} {\bf{100}} (1958), 45-92.
\bibitem{bcg} E.~Bujalance, F-J.~Cirre, and G.~Gromadzki,
A survey of research inspired by Harvey's theorem on cyclic groups of automorphisms. in \emph{Geometry of Riemann surfaces,} pg. 15-37,
London Math. Soc. Lecture Note Ser., {\bf{368}}, Cambridge Univ. Press, Cambridge, (2010).
\bibitem{gktrans} M.~Giulietti and G.~Korchm\'aros,
Algebraic curves with a large non-tame automorphism group fixing no point, {\em Trans. Am. Math. Soc.}, {\bf{362}} (2010), 5983--6001.
\bibitem{gkja} M.~Giulietti and G.~Korchm\'aros, Large p-groups of automorphisms of algebraic curves in characteristic $p$, \emph{J. Algebra} {\bf{481}} (2017), 215-249.
\bibitem{hw} W.J.~Harvey, Cyclic groups of automorphisms of a compact Riemann surface.
\emph{Quart. J. Math. Oxford Ser.} (2) {\bf{17}} (1966), 86-97.
\bibitem{Qu2012} Haipeng Qu, Sushan Yang, Mingyao Xu, Lijian An, Finite $p$-groups with a minimal non-abelian subgroup of index $p$, \emph{J. Alg.} {\bf{358}} (2012), 178-188.
\bibitem{HKT} J.W.P. Hirschfeld, G. Korchm\'aros, and F. Torres, {\it Algebraic Curves over a Finite Field}, Princeton Series in Applied Mathematics, Princeton (2008).
\bibitem{huppertI1967} B.~Huppert, \emph{Endliche Gruppen. I}, Grundlehren der Mathematischen Wissenschaften {\bf 134}, Springer, Berlin, (1967), xii+793 pp.
\bibitem{knp}G.~Korchm\'aros, G.P.~Nagy, and N.~Pace, 3-Nets realizing a group in a projective plane, \emph{J. Algebr. Comb.} {\bf{39}} (2014), 939--966.
\bibitem{LM} C.~Lehr and M.~Matignon,
Automorphism groups for p-cyclic covers of the affine line, \emph{Compos. Math.} {\bf{141}} (2005), 1213-1237.
\bibitem{MA} A.~Mach\`i, \emph{Groups, An introduction to ideas and methods of the theory of groups}, Unitext, {\bf{58}}, Springer, Milan, (2012), xiv+371 pp.
\bibitem{MR} M.~Matignon and M.~Rocher, Smooth curves having a large automorphism p-group in characteristic $p>0$, \emph{Algebra Number Theory} {\bf{2}} (2008),  887-926.
\bibitem{N} S.~Nakajima, $p$-ranks and automorphism groups of algebraic curves, \textit{Trans. Amer. Math. Soc.} \textbf{303} (1987), 595-607.
\bibitem{Qu} Qu Haipeng, Yang Sushan, Xu Mingyao, and  An Lijian,
Finite p-groups with a minimal non-abelian subgroup of index p (I),
\emph{J. Algebra} {\bf{358}} (2012), 178-188.
\bibitem{ps} R.~Pries and K.~Stevenson, A survey of Galois theory of curves in characteristic $p$, In {\emph{WIN
- Women in Numbers}}: Research Directions in Number Theory, A. C. Cojocaru, K. Lauter, R. Pries, and R. Scheidler Eds., Fields Inst. Commun., {\bf{60}}, Amer. Math. Soc., Providence,
RI, (2011), 169-191.
\bibitem{stichtenoth1973I} H.~Stichtenoth,
\"Uber die Automorphismengruppe eines algebraischen
Funktionenk{\"o}rpers von Primzahlcharakteristik. I. Eine
Absch{\"a}tzung der Ordnung der Automorphismengruppe, \emph{Arch.
Math.} {\bf 24} (1973), 527--544.
\bibitem{stichtenoth1973II} H.~Stichtenoth, \"Uber die Automorphismengruppe eines algebraischen
Funktionenk{\"o}rpers von Primzahl- charakteristik. II. Ein
spezieller Typ von Funktionenk{\"o}rpern, \emph{Arch. Math.} \textbf{24} (1973), 615--631.
\bibitem{stichtenoth1993} H.~Stichtenoth, \emph{Algebraic Function Fields and Codes},
Springer-Verlag, Berlin and Heidelberg, (1993), vii+260 pp.
\bibitem{maddenevalentini1982} R.C.~Valentini and M.L.~Madan,  A
Hauptsatz of L.E. Dickson and Artin--Schreier extensions, \emph{J.
Reine Angew. Math.} {\bf 318} (1980), 156--177.
\bibitem{Xu2008} Mingyao Xu, Lijian An, and Qinhai Zhang, Finite p-groups all of whose non-abelian proper
subgroups are generated by two elements, \emph{J. Alg.} {\bf{319}} (2008), 3603-3620.
\bibitem{zr} R.~Zomorrodian,
Nilpotent automorphism groups of Riemann surfaces,
\emph{Trans. Amer. Math. Soc.} {\bf{288}} (1985), 241-255.
\bibitem{zr1} R.~Zomorrodian, Classification of $p$-groups of automorphisms of
Riemann surfaces and their lower central series,
\emph{Glasgow Math. J.} {\bf{29}} (1987), 237-244.
\bibitem{yu} S.~Yuzvinsky, S, Realization of finite Abelian groups by nets in $P_2$, \emph{Compos. Math.} {\bf{140}}, (2004), 1614-1624.
\end{thebibliography}
\end{document}